\documentclass[twoside]{article}

\usepackage[accepted]{aistats2021}
\usepackage[utf8]{inputenc} 
\usepackage[T1]{fontenc}    
\usepackage{hyperref,amssymb}       
\usepackage{url}            
\usepackage{booktabs}       
\usepackage{amsfonts}       
\usepackage{nicefrac}       
\usepackage{microtype}      
\usepackage{soul}
\usepackage{xcolor}
\usepackage{comment} 
\usepackage{natbib}
\usepackage{tablefootnote}

\usepackage{amsmath, amsthm, amsopn, amstext, amsfonts, algorithm, algpseudocode, graphicx, mathabx}

\usepackage{multirow}

\newcommand{\cA}{\mathcal{A}}
\newcommand{\cP}{\mathcal{P}}
\newcommand{\cI}{\mathcal{I}}

\newcommand{\xr}{x^{r}}
\newcommand{\yr}{y^{r}}
\newcommand{\lr}{\lambda^{r}}
\newcommand{\xrt}{x^{r_t}}
\newcommand{\yrt}{y^{r_t}}
\newcommand{\lrt}{\lambda^{r_t}}
\newcommand{\xb}{\bar{x}}
\newcommand{\yb}{\bar{y}}
\newcommand{\lb}{\bar{\lambda}}
\newcommand{\xrn}{x^{r+1}}
\renewcommand{\st}{{\rm s.t.}}
\newcommand{\yrn}{y^{r+1}}

\newcommand{\LA}{\mathcal{L}}

\newcommand{\qadmm}{\texttt{ADMM-Q}}
\newcommand{\admmq}{\texttt{ADMM-Q}}

\newcommand{\iqadmm}{\texttt{I-ADMM-Q}}

\newcommand{\admmr}{\texttt{ADMM-R}}
\newcommand{\admms}{\texttt{ADMM-S}}

\newtheorem{theorem}{Theorem}[section]
\newtheorem{proposition}[theorem]{Proposition}
\newtheorem{lemma}[theorem]{Lemma}

\newtheorem{definition}[theorem]{Definition}
\newtheorem{remark}[theorem]{Remark}
\newtheorem{example}[theorem]{Example}
\newtheorem{assumption}[theorem]{Assumption}

\begin{document}

\runningauthor{T. Huang, P. Singhania, M. Sanjabi, P. Mitra, M. Razaviyayn}

\twocolumn[

\aistatstitle{Alternating Direction Method of Multipliers for Quantization}

\aistatsauthor{Tianjian Huang$^*$\;\; Prajwal Singhania$^\dagger$\;\; Maziar Sanjabi$^\ddag$ \;\; Pabitra Mitra$^\dagger$ \;\; Meisam Razaviyayn$^*$\\
\hspace{-0.15in}tianjian@usc.edu\;\;\;\;prajwal1210@gmail.com\;\;\; maziars@fb.com\;\;\;\; pabitra@gmail.com \;\;\;\quad  razaviya@usc.edu\bigskip}
\aistatsaddress{$^*$University of Southern California\\
$^\dagger$Indian Institute of Technology Kharagpur\\
$^\ddag$Facebook AI} 
]

\begin{abstract}
Quantization of the parameters of machine learning models, such as  deep neural networks, requires solving constrained optimization problems, where the constraint set is formed by the Cartesian product of many simple discrete sets. For such optimization problems, we study the performance of the Alternating Direction Method of Multipliers for Quantization (\admmq) algorithm, which is a variant of the widely-used ADMM method applied to our discrete optimization problem. We establish the convergence of the iterates of  \admmq{}  to certain \textit{stationary points}. In addition, our results shows that the Lagrangian function of ADMM converges monotonically. To the best of our knowledge, this is the first analysis of an ADMM-type method for problems with discrete variables/constraints. Based on our theoretical insights, we develop a few variants of  \admmq{} that can handle inexact update rules, and have improved performance via the use of ``soft projection'' and ``injecting randomness'' to the algorithm. We empirically evaluate the efficacy of our proposed approaches on two problem: 1) solving quantized quadratic optimization problems and 2) training neural networks. Our numerical experiments shows that \admmq{} outperforms other competing algorithms. 
\end{abstract}

\section{Introduction}


The fields of machine learning and artificial intelligence have experienced significant advancements in recent years. Despite this rapid growth, the extreme energy consumption of many existing machine learning models prevents their use in low-power devices.
As a solution, \textit{quantized} and \textit{binarized} training of these models have been proposed in recent years \cite{courbariaux2015binaryconnect, courbariaux2016binarized, rastegari2016xnor, 1312.6199}. This procedure requires training  a machine learning model that has low training/test error, and at the same time, low power/storage requirement. More precisely, the parameters of the machine learning model must lie in a discrete set (e.g. the weights of the neural network should be binary). The goal is to improve the energy and storage efficiency of the model by simplifying the required storage/computation in the inference phase.

To obtain accurate quantized models, a wide range of training techniques  have been proposed.
Among them, Alternating Direction Method of Multipliers (ADMM) has  recently gained popularity and resulted in training highly accurate machine learning models with super low power consumption \cite{ye2018progressive, ye2019progressive, yuan2019tp, yuan2019ultra, leng2018extremely, lin2019toward, zhang2018systematic, liu2020autocompress, ren2019admm, li2019admm}. Despite this empirical success, the theoretical understanding of  ADMM  for solving discrete optimization problems, such as training binarized neural networks, is almost non-existent. As a first step toward better understanding the behavior of this algorithm in solving discrete problems, in this paper we aim at \textit{studying the behavior of the ADMM algorithm in discrete optimization through answering the following simple yet fundamental questions}:

    \textbf{Brief Problem Description.}   Assume ADMM algorithm  is applied to a nonconvex discrete optimization problem such as training binarized/quantized neural networks. 
    \begin{itemize}
        \item Is  ADMM  guaranteed to improve the objective function over the iterates? Can we end up at a point that is worse than the initial point?
        \item What can we say about the ``limit points'' of the iterates generated by the ADMM algorithm in this discrete context?
        \item Can ADMM tolerate inexact, randomized, or stochastic computations?
        \item Is ADMM better than simple algorithms such as projected gradient descent when applied to this discrete problem? 
    \end{itemize}


The answer to the above fundamental questions is non-trivial. This lack of understanding is due to the non-monotonic behavior of the objective function through ADMM iterates as well as the highly fragile relations between the primal and dual variables in  this discrete optimization setting. 
In this paper, we (partially) answer the above questions by first showing that the \admmq{} algorithm, which is a variant of ADMM in discrete setting, indeed improves the objective over iterations. We analyze the limit points of the iterates generated by \admmq{} and show that every limit point of the iterates satisfies certain \textit{stationarity property}. 
Then, we extend our analysis to inexact and randomized update rules that happen in many practical problems such as training binarized neural networks. 
Finally, we evaluate the performance of \admmq{} and its extensions in our numerical experiments. The goal of our numerical experiments is not to obtain the best performance in a particular application or existing benchmark problems, but instead to better understand the behavior of \admmq~method. Notice that \admmq~has already been used in other papers and its efficiency (combined with other training heuristics) has been established in the literature  for different problems \cite{leng2018extremely, lin2019toward, yuan2019tp, yuan2019ultra}.
Moreover, to obtain a better understanding of \admmq, we avoid using heuristics  such as  Straight-through Estimators,
scaling factor, not binarizing last layer, playing with the architecture, which has been used in other papers \cite{bengio2013estimating, rastegari2016xnor, darabi2018bnn+, tang2017train}. 
While these heuristics (combined with exact tuning of many parameters) can significantly improve the performance of the method, they make the scientific study of the core \admmq~algorithm almost impossible by bringing a lot of other not well-understood approaches to the table. Thus, in our numerical experiments, instead of aiming for the best possible performance, obtained by using multiple heuristics, we only focus on the empirical performance of the core quantization algorithm. 


\subsection{State of the art} \label{sec:StateOfArt}
This paper studies the behavior of the ADMM algorithm when applied to nonconvex discrete optimization problems. This is closely tied to the previous studies on the ADMM algorithm and training quantized machine learning models. Here we briefly review some of the existing works in each of these two categories:

\textbf{Quantized machine learning models.} In recent years, there have been numerous works on the quantization of machine learning models---specifically neural networks. 
One of the first works towards this was BinaryConnect \cite{courbariaux2015binaryconnect} which used the ``Straight Through Estimator'' (STE)~\cite{bengio2013estimating} to provide a ``from-scratch'' training method with binary weights. BinaryNet~\cite{courbariaux2016binarized} extended upon this idea to binarize both weights and activations, replacing complex convolutions with simpler bit-wise operations and significantly reducing the computational complexity. 
These works performed very well on smaller datasets like MNIST, SVHN and CIFAR-10, and provided an important direction for compression of neural networks. However, their performance on ImageNet~\cite{deng2009imagenet} classification was  poor. XNOR-Net~\cite{rastegari2016xnor} was one of the first works to improve binarized CNNs for ImageNet classification by using scaling factors, that trade-off compression with accuracy. 
DoReFa-Net~\cite{zhou2016dorefa} further extended the idea of binarization (using the sign function) to gradients as well. They also generalized the method to create networks with arbitrary bit-widths for weights, activations and gradients. ABC-Net \cite{lin2017towards} improved upon the ideas from XNOR-Net by using multiple binary weights to approximate the full precision weights (instead of scaling factors) and using multiple binary activations. 
These changes showed  that performance like that of XNOR-Net can be achieved without the scaling factors. \cite{tang2017train} introduced seemingly small but impacting changes to improve accuracy, one of which was the use of a regularization function: $|1-W^2|$ that carried on to further works. BNN+~\cite{darabi2018bnn+} brought about yet another performance boost by careful regularization strategies and replacing the plain STE with a ``SignSwish'' activation, a modified version of the Swish-like activation~\cite{ramachandran2017searching}. \cite{yin2019understanding} provide key theoretical justification to the use of STE by showing a positive correlation between the true and the estimated ``coarse'' gradient obtained through STE chain rule. 

\textbf{ADMM algorithm.} ADMM is an optimization algorithm that combines the decomposability of dual ascent with the the superior convergence guarantees of the method of multipliers. The algorithm, which is believed to be first introduced by \cite{glowinski1975approximation} 
and \cite{gabay1976dual}, can be shown to be equivalent to the Douglas-Rachford splitting algorithm \cite{douglas1956numerical}. \cite{boyd2011distributed} provides a comprehensive overview of the method. Recently, ADMM has sparked the interests of many researchers due to its simplicity, theoretical convergence rates, and parallelization capabilities. 
The extensibility of ADMM to inexact proximal updates and non-convex problems make it appealing for a lot of problems in machine learning. 
\cite{hong2016convergence}~is perhaps the first work that extended the analysis of ADMM to nonconvex problems and showed its convergence to first-order stationary points. This analysis is later strengthened in~\cite{hong2018gradient} by showing the convergence of ADMM iterates to second-order stationary points. 
Another interesting work by \cite{wang2019global} analyzed the convergence of ADMM for nonconvex and possibly nonsmooth objectives and showed that ADMM, applied to many statistical problems, is guaranteed to converge. In the optimization society, the behavior of ADMM when applied to problems with nonconvex objective functions have also been studied in other regimes such as multiaffine constraints~\cite{gao2020admm}, dynamically changing convex constraints~\cite{zhang2020online}, finite-sum objective functions, inexact and asynchronous update rules~\cite{hong2017distributed, zhang2018proximal}, to name just a few. {\color{black} 
 \cite{NIPS2014_5612} generalizes ADMM to Bregman ADMM (BADMM), which allows the choice of different Bregman divergences to exploit the structure of problems.} ADMM has also been used as heuristics to solve mixed-integer quadratic programming. \cite{takapoui2020simple} proposed an ADMM based algorithm approximately solving convex quadratic functions over the intersection of affine and separable constraints. The Deep Learning community has been no exception to this increased interest in ADMM. 
\cite{wang2019admm} provided global convergence guarantees for an ADMM-based optimizer for deep neural networks. \cite{ye2018progressive} used ADMM to devise an effective weight-pruning technique in DNNs for better compression.

{\color{black} With the rising interest in quantization for neural network compression, several works have tried  ADMM-based approaches. \cite{leng2018extremely} were among the first to use an ADMM formulation for weight quantization (not activations) and demonstrated extremely superior results on ImageNet classification. \cite{zhang2018systematic} proposed a systematic DNN weight pruning framework using ADMM. \cite{ye2019progressive, lin2019toward, zhang2018systematic} extended the work~\cite{zhang2018systematic} by proposing a progressive multi-step approach that not only leads to a  better performance, but also can be applied to weight binarization. TP-ADMM \cite{yuan2019tp} used powerful practical improvements to break the training procedure into optimized stages and extend the formulation for binarizing both weights and activations with the state of the art results. \cite{liu2020autocompress} proposed an automatic structured pruning framework, adopting ADMM based algorithm, which boosted the compression ratio to an even higher level. Several works investigated the implementation of ADMM based weight pruning algorithms on hardware level. \cite{yuan2019ultra, ren2019admm} explored the idea of algorithm-hardware co-design framework  using ADMM. \cite{li2019admm} showed ADMM based weight pruning achieved significant
storage/memory reduction and speedup in mobile devices with negligible
accuracy degradation.  In spite of these promising empirical results, the theoretical understanding of ADMM with respect to quantization is still close to non-existent. } 

\section{Problem Formulation}


Consider the following discrete optimization problem:
\begin{equation}
\label{eq:original}
\begin{aligned}
&\underset{x}{\text{min}}
 \;\;f(x), 
\quad \st
\;\;\;\; x \in \mathcal{A}=\left\{a_1, a_2, \ldots, a_n\right\}\subseteq \mathbb{R}^d
\end{aligned}
\end{equation}
where $\mathcal{A}$ is a discrete subset of $\mathbb{R}^d$. One approach for solving this  problem is to sweep across all values in~$\mathcal{A}$ and find the optimum point. While this approach results in finding the global optimal solution(s), it is not  practical  in the quantization procedures of machine learning models. 
In particular, in this application, the  set $\mathcal{A}$ is a discrete grid defined over the space of neural network parameters. Hence, $n = |\mathcal{A}|$ is exponential in the dimension~$d$ and it is computationally impossible to sweep over all values of $\cA$. 
While the size of the set $\mathcal{A}$ can be  very large, we make an assumption that the projection to the set $\cA$ can be done efficiently. To state our assumption clearly, let us formally define the projection operator followed by two clarifying examples.

\begin{definition}
For any finite set $\mathcal{A}$, the projection of a point $x$, defined as $\mathcal{P}_{\mathcal{A}}(x)$, is a  point $x_p = \arg\min_{a\in \mathcal{A}} \|x-a\|^2$. If the set~$\arg\min_{a\in \mathcal{A}} \|x-a\|^2$ is non-singleton, we choose an element in the set with the smallest lexicographical value\footnote{We can break the tie in different ways. We can also pick one of the points in the set~$\arg\min_{a\in \mathcal{A}} \|x-a\|^2$ uniformly at random. This choice will make our results to hold with probability one.}.
\end{definition}

\begin{assumption}
\label{assumption: 2.2}
Projection to the set~$\mathcal{A}$ can be done in a computationally efficient manner.  
\end{assumption}

\begin{example} \label{Ex:Binary}
Suppose~$\cA = \{-1,+1\}^d$ in \eqref{eq:original} with $|\cA| = 2^d$. One can verify that $\cP_\cA (x) = {\rm sign}(x) = (\bar{x}_1,\ldots,\bar{x}_d) \in \mathbb{R}^d$ where $\bar{x}_i = +1$ if $x_i \geq 0$ and $\bar{x}_i = -1$ if $x_i<0$. Thus, despite the exponential size of the set $\cA$, the projection operator can be  computed efficiently.
\end{example}

\begin{example}\label{Ex:Quantized}
Assume $\cA = \{x \in  \mathbb{Z}^d\; | \; a\leq x \leq b\}$ with $a,b \in \mathbb{R}^d$ and $\mathbb{Z}$ being the set of integer numbers. Due to the Cartesian product structure of the set~$\cA$, one can verify that $\cP_\cA(x) = (\bar{x}_1,\ldots,\bar{x}_d)$ with $\bar{x}_i = b_i$ if $x_i>b_i$, $\bar{x}_i = a_i$ if $x_i<a_i$, and  $\bar{x}_i = {\rm round}(x_i)$ if $a_i \leq x_i \leq b_i$. Thus, the projection operator can be computed efficiently despite the exponential size of $\cA$.
\end{example}

The above two examples are the constraint sets that appear in the quantization/binarization of machine learning models. Next, we describe the ADMM algorithm for solving optimization problem~\eqref{eq:original}.

\section{Alternating Direction Method of Multipliers for Quantization (\admmq)}
\subsection{Review of ADMM} \label{subsec:ADMMReview}
ADMM
aims at solving linearly constrained optimization problems of the form
\begin{align}
    \min_{w,z} \;\;  h(w) + g(z)\quad \quad 
    \st \quad  Aw + Bz = c, \nonumber
\end{align}
where $w\in \mathbb{R}^{d_1}, z\in \mathbb{R}^{d_2}$, $c \in \mathbb{R}^{k}$, $A \in \mathbb{R}^{k \times d_1}$, and $B \in \mathbb{R}^{k \times d_2}$. By forming the augmented Lagrangian function 
\begin{multline*}
\mathcal{L}(w,z,\lambda) \triangleq h(w) + g(z) + \langle \lambda , Aw + Bz - c\rangle  \\+ \frac{\rho}{2} \|Aw + Bz - c\|_2^2,
\end{multline*}
each iteration of ADMM applies  alternating minimization to the primal variables and  gradient ascent  to the dual variables. More precisely, at iteration~$r$, ADMM uses the update rules:
\begin{align}
    & \textrm{Primal Update:} &&w^{r+1}  = \arg\min_{w} \mathcal{L}(w,z^r,\lambda^r),\\
    & &&z^{r+1}  = \arg\min_{z} \mathcal{L}(w^{r+1},z,\lambda^r)\nonumber\\
    & \textrm{Dual Update:} &&\lambda^{r+1}  = \lambda^r + \rho \left( Aw^{r+1} + Bz^{r+1} - c\right).  \nonumber
\end{align}
As discussed in section~\ref{sec:StateOfArt},  this algorithm has been well-studied for continuous optimization. Next, we discuss how this algorithm can be used in the discrete optimization problem~\eqref{eq:original}.

\subsection{Description of \admmq}
\label{subsec: admmq_form}
In order to apply ADMM algorithm to the quantization problem~\eqref{eq:original}, we first re-write~\eqref{eq:original} as
\begin{equation} \label{eq:ReformulationForADMM}
\min_{x} \quad f(x) + \mathcal{I}_\cA (y) \quad \quad  \st \quad x = y,
\end{equation}
where $\cI_\cA (y) = 0 $ if $y \in \cA$, and $\cI_\cA (y) = +\infty $ if $y \notin \cA$. 
Following the steps of regular ADMM in section~\ref{subsec:ADMMReview}, we can update the primal and dual variables alternatingly. The resulting algorithm, which is called Alternating Direction Method of Multipliers for Quantization (\admmq), is summarized in Algorithm~\ref{alg: ADMM exact}. The details of the derivation of this algorithm can be found in appendix~\ref{Apx:ExactAdmmDetails}. 
Step~\ref{xupdate} in this algorithm requires solving an unconstrained optimization problem. In our setting, as we will see later, when $\rho$ is chosen large enough, the function $\mathcal{L}(x,y^{r+1},\lambda^r)$ is strongly convex in $x$. Thus solving this problem is assumed to be possible for now. We later relax step~\ref{xupdate} to inexact update rule.
\begin{algorithm}[]
	\caption{\qadmm} 
	\label{alg: ADMM exact}
	\begin{algorithmic}[1]
	    \State {\textbf{Input}}: Constant~$\rho>0$; initial points $x^0 = y^0 \in \cA$, $\lambda^0  \in \mathbb{R}^d$
		\For{$r = 0,1,2,\ldots$}
		\State \textbf{Update $y$}: \quad  $y^{r+1}=\mathcal{P}_\mathcal{A}(x^r+\rho^{-1}\lambda^r)$ \label{step:ADMMQYupdate}
		\State \textbf{Update $x$}: \quad  $x^{r+1}=\arg\min_x \mathcal{L}(x, y^{r+1},\lambda^r)$ \label{xupdate}
		\State \textbf{Update $\lambda$}:\quad  $\lambda^{r+1}=\lambda^r+\rho(x^{r+1}-y^{r+1})$
		\EndFor

	\end{algorithmic}
\end{algorithm}

\subsection{Convergence Analysis of \admmq} 
\label{sec:ADMMQ}
In order to analyze the behavior of~\admmq, we make the following assumptions on $f$:
\begin{assumption}
\label{assumption: lowerbdd}
The function $f$ is lower bounded on $\mathcal{A}$. That is, $-\infty < f_{\min} \triangleq \min_{a\in\mathcal{A}} f(a)$.
\end{assumption}
\begin{assumption}
\label{assumption: lipschitz}
The function f is differentiable and its gradient is $L_f$--Lipschitz, i.e., 
$$\| \nabla f(x) - \nabla f(y)\|\leq  L_f \| x - y \|, \; \forall x, y \in \mathbb{R}^d. $$
\end{assumption}
\begin{assumption}
\label{assumption: hessianbddbelow}
There exist a constant $\mu\geq 0$ such that $f$ is $\mu$-weakly convex, i.e. $f(x)+\frac{\mu}{2}\|x\|^2$ is convex.
\end{assumption}
 When $f$ is twice continuously differentiable, it is easy to verify that $\mu\leq L_f$. However, defining these two constants separately will allow us to get tighter bounds for the cases that these two constants are different.
Let us also state a few useful lemmas that will help us understand the behavior of \admmq. The proofs of these lemmas are relegated to appendix~\ref{sec:ProofsofADMMQ}.
\begin{lemma}
\label{lemma: lowerbound}
If $\rho\geq L_f$, we have
$ 
\mathcal{L}(\xr,\yr,\lr)\geq f(\yr) \geq  f_{\min},  \quad \forall r\geq 1.$
\end{lemma}
\begin{lemma}\label{lemma: decrease}
Define $\sigma(\rho) \triangleq \rho-\mu$.  We have 
\begin{multline}
    \mathcal{L}(x^{r+1},y^{r+1},\lambda^{r+1})-\mathcal{L}(x^r,y^r,\lambda^r)\\
    \leq (\rho^{-1}L_f^2-\frac{\sigma(\rho)}{2})\left\|x^{r+1}-x^r\right\|^2.
\end{multline}
\end{lemma}
This lemma states that by choosing $\rho$ large enough so that
$
        \rho^{-1} L_f^2-\frac{\sigma(\rho)}{2}<0,
$
we ensure the decrease of the augmented Lagrangian function at each iteration\footnote{When $f$ is convex,  $\mu= 0$ and hence~$\sigma(\rho) = \rho$.  Thus choosing $\rho>\sqrt{2}L_f$ suffices to ensure the decrease of the augmented Lagrangian function. For the general nonconvex twice differentiable functions, choosing $\rho>2L_f$ will imply that $\rho^{-1} L_f^2-\frac{\sigma(\rho)}{2}<0$, and hence  the decrease is guaranteed by Lemma~\ref{lemma: decrease}.}.
This property combined with Lemma~\ref{lemma: lowerbound} implies that $f(y^r) \leq \mathcal{L} (x^r,y^r,\lambda^r)\leq \mathcal{L} (x^0,y^0,\lambda^0) =f(y^0)$. That is, \admmq{} cannot output a point  worse than the initial point.
Next, we use these lemmas to analyze the limitting behavior of the iterates of \admmq.  To do that, let us first define the following stationarity concept.

\begin{definition}
\label{def: stationarity}
We say a point $\xb$ is a $\rho-$stationary point of the optimization problem~\eqref{eq:original} if
\begin{equation}
\nonumber
\bar{x} \in \arg\min_{a\in\mathcal{A}}\|a-(\xb-\rho^{-1}\nabla f(\xb))\|.
\end{equation}
\end{definition}
In other words, the point~$\bar{x}$ cannot be locally improved using projected gradient descent with step-size $\rho^{-1}$.   Unlike the usual definitions of stationarity for convex constraints, our definition of stationarity  depends on the constant $\rho$. Denoting the set of $\rho$-stationary solutions with $\mathcal{T}_\rho$,  it is easy to see that $\mathcal{T}_{\rho_1}\subseteq\mathcal{T}_{\rho_2}$ when $\rho_1\leq \rho_2$. Thus, in general we would want to have $\rho$ as small as possible. The following lemma justifies the definition of $\rho$-stationary.
{
\color{black}
\begin{lemma}
\label{lemma: optimality}
Assume $x^\star$ is an optimal solution to problem~\eqref{eq:original}, then $x^\star$ is a $\rho$-stationary point for any $\rho\geq L_f$.
\end{lemma}

Our $\rho$-stationarity definition (Definition~\ref{def: stationarity}) is a natural extension of the continuous setting. It is also closely related to stationarity defined in proximal gradient methods, e.g., see~\cite{drusvyatskiy2018error, kadkhodaie2014linear}, in particular when the proximal operator is associated with an indicator function. Note that, our $\rho$-stationary definition is a non-trivial necessary condition for optimality. Also when ${\rho<L_f}$, such stationary points may not exist. See Example~\ref{example:nonexist} in appendix~\ref{sec:ProofsofADMMQ}. 
{
\color{black}
\begin{remark}
    Because of the fact that Definition~\ref{def: stationarity} is a natural extension of the continuous case, it is straightforward to prove the convergence of Projected Gradient Descent (PGD)\footnote{Each step of PGD comprises of performing a gradient step and then projecting to the feasible set.} algorithm to such stationary set; see appendix~\ref{sec:PGDConvergence} for more details.
\end{remark}
}
\begin{theorem} \label{Thm:ADMMQConvergence}
Assume that $f$ satisfies  Assumptions~\ref{assumption: lowerbdd}, \ref{assumption: lipschitz} and \ref{assumption: hessianbddbelow}. Assume further that $\rho$ is chosen large enough so that $\rho^{-1} L_f^2-\frac{\sigma(\rho)}{2}<0$. Let $(\xb, \yb, \lb)$ be a limit point of the \admmq~algorithm. Then $\xb$ is a $\rho$--stationary point of the optimization problem~\eqref{eq:original}.
\end{theorem}
{
\color{black}
\begin{remark}
    The previous convergence results for non-convex ADMM, \cite{li2015global}, do not apply to our specific setting. For our problem~\eqref{eq:original}, the stationarity notion defined in equation~(4) of~\cite{li2015global} is satisfied for every feasible point since the sub-differential set of every feasible point contains~$0$ (when the feasible set is discrete and finite size). Thus, any feasible point is a stationary point according to the stationary notion in~\cite{li2015global} (see equation~(4)). Thus the convergence results and inequalities in~\cite[Theorem 1]{li2015global} would be vacuous in our setting. 
\end{remark}

}
}



{\color{black}
\begin{remark}
The convergence results presented in \cite{wang2019global} do not apply to our setting. This is due to the fact that \cite{wang2019global}~uses Lipschitz sub-minimization paths assumption (Assumption A3). If we specialize their assumption to our setting,  their assumption requires that the mappings $H(u) = \arg\min_y f(x) + \mathcal{I}_{\mathcal{A}} (y) \;\textrm{s.t.} \; y=u$ 
 and  $F(u) = \arg\min_x f(x)\;\textrm{s.t.} \; y=u $  are well-defined and Lipschitz continuous.  Clearly, Both of these assumptions do not hold in our setting due to non-convexity (and disconnected nature) of the set $\mathcal{A}$. Moreover, regarding global convergence, Theorem~1 and 2 in \cite{wang2019global}~use KL condition (after introducing indicator functions). These assumptions also do not hold in our setting.
\end{remark}}




While Theorem~\ref{Thm:ADMMQConvergence} establishes the convergence of~\admmq,  this algorithm is far from its inexact version implemented in practice. Next, we analyze the inexact version of \admmq~which is used most often in practice and in particular in training binarized neural networks. 
\section{Inexact \qadmm~(\iqadmm)}


Updating the variable~$x$ in~\qadmm~requires finding the minimizer of $\mathcal{L}(\cdot, y^{r+1}, \lambda^r)$; see step~\ref{xupdate} in Algorithm~\ref{alg: ADMM exact}.  
Although~$\mathcal{L}(\cdot, y^{r+1}, \lambda^r)$ is strongly convex when $\rho>\mu$, finding the exact minimizer might not be practically possible.  In practice, we apply iterative methods such as (stochastic) gradient descent to obtain an approximate solution $x^{r+1} \approx \arg\min_x \mathcal{L}(x ,y^{r+1}, \lambda^r)$. In this section, we show that \admmq{} algorithm converges under such an inexact update rule. More precisely, instead of the exact update rule in step~\ref{xupdate} of Algorithm~\ref{alg: ADMM exact}, we choose a $\gamma$--approximate point~$x^{r+1}$ that satisfies 
\begin{equation} \label{eq:InexactUpdate}
\|x^{r+1}-x^{r+1}_\star\|\leq \gamma \min\;\{\|x^{r+1}-y^{r+1}\|,\|x^{r+1}-x^r\| \},
\end{equation}
for some positive constant $\gamma$. Here  $x_\star^{r+1} \triangleq \arg\min_x \mathcal{L}(x ,y^{r+1}, \lambda^r)$ is the exact minimizer. 
The  resulting inexact ADMM algorithm, dubbed $\iqadmm$, is summarized in Algorithm~\ref{alg: inexact_admm}. Notice that when $\gamma = 0$, this inexact algorithm reduces to the exact \admmq{} algorithm.

\begin{algorithm}[]
 	\caption{\iqadmm  
 	} 
	\label{alg: inexact_admm}
	\begin{algorithmic}[1]
	    \State {\textbf{Input}}: Constants $\rho, \gamma>0$; initial points $x^0 = y^0 \in \cA$, $\lambda^0  \in \mathbb{R}^d$
		\For{$r = 0,1,2,\ldots $} 
		\State \textbf{Update $y$}: $y^{r+1}=\mathcal{P}_\mathcal{A}(x^r+\rho^{-1}\lambda^r)$
		\State \textbf{Update $x$} by finding a point $x^{r+1}$ satisfying~\eqref{eq:InexactUpdate} 
		\State \textbf{Update $\lambda$}: $\lambda^{r+1}=\lambda^r+\rho(x^{r+1}-y^{r+1})$
		\EndFor
	\end{algorithmic}
\end{algorithm}

Similar inexactness measures have previously been used in the literature; see, e.g., \cite{li2018federated,reddi2016aide}. Notice that since $\mathcal{L}(x,y,\lambda)$ is strongly convex in $x$, gradient descent algorithm requires only $O(\log (1/\gamma))$ iterations to find a $\gamma$-approximate solution. Hence, in practice, we do not need to run many iterations of gradient descent. Next, we present our convergence result for \iqadmm.

\begin{theorem}
Assume that $f$ satisfies  Assumptions~\ref{assumption: lowerbdd}, \ref{assumption: lipschitz} and \ref{assumption: hessianbddbelow}. Also assume that  the iterates of  \iqadmm{} are bounded, and the constant $\rho$ and $\gamma$ are chosen such that 
\begin{equation}
    \nonumber
    \frac{2L_f^2+8(\rho+L_f)^2\gamma^2}{\rho}+ \frac{\gamma^2(\rho+L_f)-(1-\gamma)^2\sigma(\rho)}{2}
    <0,
    \label{eq: inexact_cond_gamma}
\end{equation}
with $\sigma(\rho) = \rho-\mu$. Then, for any limit point $(\xb, \yb, \lb)$ of the iterates, $\xb$ is a $\rho$--stationary point of~\eqref{eq:original}.
\end{theorem}


One can verify that the inequality above always holds for $\rho = 6L_f$ and $\gamma \leq 0.1$. However, depending on various trade-offs, we may choose different values of $\gamma$ and $\rho$. 

\vspace{0.1cm}

\begin{remark}
    In practice, checking  condition~\eqref{eq:InexactUpdate} may  be impossible since $x_\star^{r+1}$ is not known exactly. To resolve this issue, notice that the  strong convexity of  $\mathcal{L}(\cdot, y^{r+1},\lambda^r)$ implies that $\sigma(\rho) \|x - x_\star^{r+1}\| \leq \|\nabla_x \mathcal{L}(x,y,\lambda) \|$. Hence, we can use the following checkable sufficient condition instead of \eqref{eq:InexactUpdate}:
\begin{multline*}
\| \nabla_x \mathcal{L}(x^{r+1},y^{r+1},\lambda^r) \| \\
\leq \rho \gamma \min\;\{\|x^{r+1}-y^{r+1}\|,\|x^{r+1}-x^r\| \}.
\end{multline*}

\end{remark}


\section{Injecting Randomness to the Algorithm}


The analyses in the previous sections only show that the algorithm converges to a stationary solution of the form defined in Definition \ref{def: stationarity}. 
As mentioned earlier,  our stationary set  includes more points as $\rho$ increase. Thus, to obtain a  point satisfying stronger stationary condition, we need to pick the smallest possible~$\rho$. 
However, reducing the value of $\rho$ beyond certain value results in instability and divergence in \qadmm, as suggested by our theory and  numerical experiments. 
Another approach that has been utilized in practice to escape  spurious stationary solutions is the use of randomness/noise in the algorithm~\cite{jin2017escape, lu2019snap, xu2018first, allen2018neon2, barazandeh2018behavior, lu2019pa}. 
In order to inject randomness to our algorithm, we propose the following step at each iteration~$r$: draw a set of (potentially correlated) Bernoulli random variables $m^r= \{m^r_1, m^r_2, \ldots, m^r_d\}$. 
Each $m_i^r$, corresponds to the coordinate $i$ in vector $y$ with ${\text{Prob}(y_i^r=1) = p_i^r>0}$. Then, we update $y_i$ in iteration $r$ if and only if $m_i^r=1$. This variant of \admmq, which we denote by \admmr, is presented in Algorithm~\ref{alg: ADMM random}. 
The convergence result of this algorithm is stated in Theorem~\ref{thm:ConvRandom}. The proof of this result follows the same steps as in the ones in Theorem~\ref{Thm:ADMMQConvergence}, and hence we omit the proof here.
 
\begin{algorithm}[]
	\caption{\admmr} 
	\label{alg: ADMM random}
	\begin{algorithmic}[1]
	    \State {\textbf{Input}}: Constants $\rho, \gamma>0$; initial points $x^0 = y^0 \in \cA$, $\lambda^0  \in \mathbb{R}^d$; the sequence $\{p_i^r\}_{i,r}\geq \alpha>0$.
		\For{$r = 0,1,2,\ldots $} 		
		\State \textbf{Generate $m$}: $m^r= \{m^r_1, m^r_2, \ldots, m^r_d\}$
		\State \textbf{Compute $\hat{y}$}: $\hat{y}^{r+1}=\mathcal{P}_\mathcal{A}(x^r+\rho^{-1}\lambda^r)$
		\State \textbf{Update $y$}: $y_i^{r+1}=m_i^r\hat{y_i}^{r+1}+(1-m_i^r) y_i^r$, \; $\forall~i = 1,\ldots,d$
		\State \textbf{Update $x$}: $x^{r+1}=\arg\min_x \mathcal{L}(x, y^{r+1},\lambda^r)$
		\State \textbf{Update $\lambda$}: $\lambda^{r+1}=\lambda^r+\rho(x^{r+1}-y^{r+1})$
		\EndFor
	\end{algorithmic}
\end{algorithm}


\begin{theorem}
\label{thm:ConvRandom}
    Assume that the constraint set~$\mathcal{A}$ in~\eqref{eq:original} is a Cartesian product of simple coordinate-wise sets of scalers. Then, under the same set of assumptions as in Theorem~\ref{Thm:ADMMQConvergence}, every iterate of the \admmr ~algorithm is a $\rho$-stationary point of~\eqref{eq:original}.
\end{theorem}

Notice that the convergence of this algorithm requires that the  set $\mathcal{A}$ to be of  the Cartesian product form. 
This assumption is necessary since the coordinates of $y$ is updated separately; see~\cite{powell1973search, bertsekas1997nonlinear, razaviyayn2013unified} for necessity of such an assumption in the presence of coordinate-wise update rule. Having said that, the constraint sets in the quantization context satisfy this assumption as illustrated in Example~\ref{Ex:Binary} and Example~\ref{Ex:Quantized}.



\section{\admmq{} with Soft Projection (\admms)}


Step~\ref{step:ADMMQYupdate} in \admmq{} algorithm requires projection to the discrete set~$\cA$. Such a projection is non-continuous which may result in instabilities in the algorithm. As a solution, we can use ``soft projection'' in ADMM algorithm. To obtain such soft projections, we start by replacing the indicator function $\mathcal{I}_\cA(\cdot)$ in the objective function with a soft indicator function defined below.



\begin{definition}
Given a finite set $\mathcal{A} \subseteq \mathbb{R}^d$, we define the Soft Indicator Function $\mathcal{S}_\cA: \mathbb{R}^d \mapsto \mathbb{R}$ as
$$\mathcal{S}_\cA(x)=\displaystyle{\min_{a\in\cA}}\|x-a\|_2.$$
\end{definition}
Replacing the indicator function~$\cI_\cA(\cdot)$ with the soft indicator function $\mathcal{S}_\cA$ in \eqref{eq:ReformulationForADMM}, we obtain 
\[
\min_{x} \quad f(x) + \beta \mathcal{S}_\cA (y) \quad \st \quad x = y,
\]
where $\beta>0$ is some given constant. Following the steps of ADMM, we obtain the ADMM algorithm with soft projections (\admms), which is summarized in Algorithm~\ref{alg: ADMM soft}. The details of the derivation of this algorithm is summarized in appendix~\ref{app: derivation_algo_4}.


\begin{algorithm}[]
	\caption{\admms} 
	\label{alg: ADMM soft}
	\begin{algorithmic}[1]
	    \State {\textbf{Input}}: Constant~$\rho>0$, $\beta>0$; initial points $x^0 = y^0 \in \cA$, $\lambda^0  \in \mathbb{R}^d$ 
		\For{$r = 0,1,2,\ldots $} 	
		\State \textbf{Compute}: $z^{r+1}=x^r+ \rho^{-1} \lambda^r$,  $\widetilde{z}^{r+1}=\mathcal{P}_\cA(z^{r+1})$ and $z_d=\widetilde{z}^{r+1}-z^{r+1}$
		\State \label{y_update_soft} \textbf{Update $y$}: \\ \quad\quad $y^{r+1}=\left\{
		\begin{array}{cl}
		z^{r+1} + \dfrac{\rho^{-1}\beta z_d}{\|z_d\|_2} &  ,\;\rho^{-1}\beta \leq \|z_d\|_2  \\
		\widetilde{z}^{r+1} & ,\;\rho^{-1}\beta >  \|z_d\|_2
		\end{array}
		\right.$ 
		\State \textbf{Update $x$}: $x^{r+1}=\arg\min_x \mathcal{L}(x, y^{r+1},\lambda^r)$
		\State \textbf{Update $\lambda$}: $\lambda^{r+1}=\lambda^r+\rho(x^{r+1}-y^{r+1})$
		
		\EndFor

	\end{algorithmic}
\end{algorithm} 
As shown in appendix~\ref{app: derivation_algo_4}, this algorithm coincides with \admmq{} if $\beta$ is chosen large enough. However, for small values of $\rho$, this algorithm results in a different trajectory. In this case, while the iterates of the algorithm does not necessarily converge to the set~$\cA$,  the $y$ iterates are kept close to set $\cA$. Moreover, as shown in appendix~\ref{app: derivation_algo_4}, the augmented Lagrangian function  is monotonically decreasing and it converges. Finally, we would like to mention that similar hard and soft indicators have been used before for sparse signal recovery through soft and hard thresholding operators~\cite{donoho1995noising, blumensath2008iterative}.




\section{Numerical Experiments}
\label{sec: numerical_experiments}


We empirically evaluate the performance of the proposed algorithms in the following two problems:  1) Solving quadratic optimization problems with integer constraints. 2) Training  quantized neural networks. {\color{black}The link to code is available in appendix~\ref{app:code}.} 



\subsection{Numerical Experiment on Quadratic Optimization with Integer Constraints}

In this experiment, we use the presented algorithms (\admmq~and its variants) to solve the optimization problem
\begin{equation}
    \min_{x} \;\; \frac{1}{2}x^\top Q x + b^\top x \quad \st \;\; x\in \cA \triangleq v\mathbb{Z}^d,
\end{equation}
for some given $Q\in \mathbb{R}^{d\times d}$, $b\in \mathbb{R}^{d}$, and $v\in \mathbb{Z}^+$. Here, the constraint set enforces that the solution should be an integer number which is a multiple of $v$.
We generate matrix $Q$ via the equation~$Q = \widetilde{Q}^\top \widetilde{Q} + \widetilde{q} \widetilde{q}^\top$,
where $\widetilde{Q}_{ij}\sim N(0, 1),\;\widetilde{q}_{i}\sim N(0,\sigma_{\widetilde{q}}^2),\;1\leq i, j\leq d$. Note that the Lipschitz constant of the objective function (parameter~$L_f$ in the previous sections) can be adjusted through changing $\sigma_{\widetilde{q}}^2$.
We compare the performance of projected gradient gradient descent~(PGD), GD+Proj, \admmq, \admms, and \admmr~for different values of $d$ and $\sigma_{\widetilde{q}}^2$ (see appendix~\ref{app: simulation_quad_all} for more details). The PGD algorithm is defined through the iterative update rule~$x^{r+1} = \cP_{\cA} (x^r - \rho^{-1} \nabla f(x^r))$. The ``GD+Proj'' algorithm, runs gradient descent to find the global optimum of unconstrained problem, then it projects the final solution onto the feasible set $\cA$.

For each problem instance, we run each algorithm initialized at the same random point for 30,000 iterations (except 100,000 iterations for PGD to make sure it is convergent). 
The best objective value over the last 50 iterations of the algorithm will be recorded as the result of each run. We repeated this procedure for 50 different initilizations, and compute the median, $25\%$ quartile and $75\%$ quartile over 50 runs. 
We use the best hyper-parameter for each algorithm by median, and report the median, $25\%$ quartile and $75\%$ quartile. The list of hyper parameters used can be found in appendix~\ref{app: simulation_quad_all}. 

\textbf{Results.}
We only report our results for $(v,d,\sigma_{\widetilde{q}}^2)=(8,16,30)$ here. More simulations can be found in appendix~\ref{app: simulation_quad_all}.
Figure~\ref{fig:Quantiles} shows the performance of the studied algorithms for five different problem instances. Each point on x-axis represents one problem instance; and y-axis is the final obtained objective value. As expected, \admmq~outperforms PGD and GD+Proj with   large margins. 
We also observe that both \admms{} and \admmr{} have better median final objective values than \admmq. In addition, the final objective value  has a smaller  variance in these two algorithms. More importantly, the median tends to overlap with the $25\%$ quantile, i.e., the objective of at least 25 runs are almost the same as the minimum objective over 50 runs. 
\begin{figure}[]
    \centering
    \resizebox{0.355\textwidth}{!}{%
    \includegraphics[] {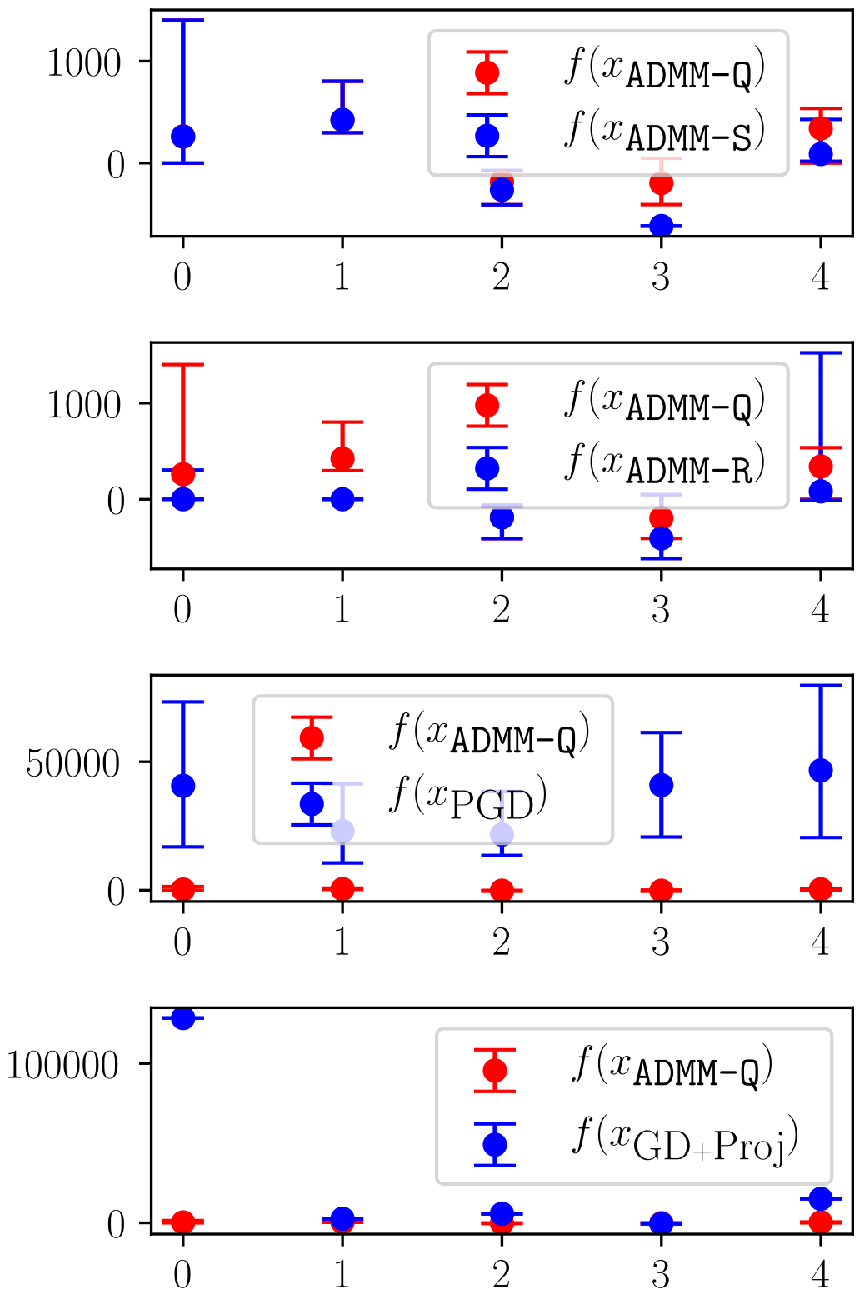}
    }
    \caption{Performance of \admmq, \admms, \admmr~ and PGD on different problem instances}
    \label{fig:Quantiles}
\end{figure}

To better understand the performance gap between different algorithm for the same initialization, we conducted one additional experiment: we generated 5 instance of $Q$ and $b$, and for each instance we ran 50 different random initialization, resulting in 250 total runs. We recorded the final  objective value by each algorithm.  Then we   computed the differences between the objective values obtained by two algorithms for the same initialization. We plot the histograms of these  differences in~Figure~\ref{fig:hist_all}. In this plot, $f(x_{\admmq})$ denotes the final objective value obtained by \admmq{} algorithm (similar notation is used for other algorithms). 
Our histogram plot suggests that \admms{} and \admmr{} outperform \admmq{} for almost all 250 runs. It also shows that PGD performs much worse than \admmq{} or its variants. We also observe that \admmr{} slightly outperforms \admms.


\begin{figure}
    \centering
    \resizebox{0.355\textwidth}{!}{%
    \includegraphics[]{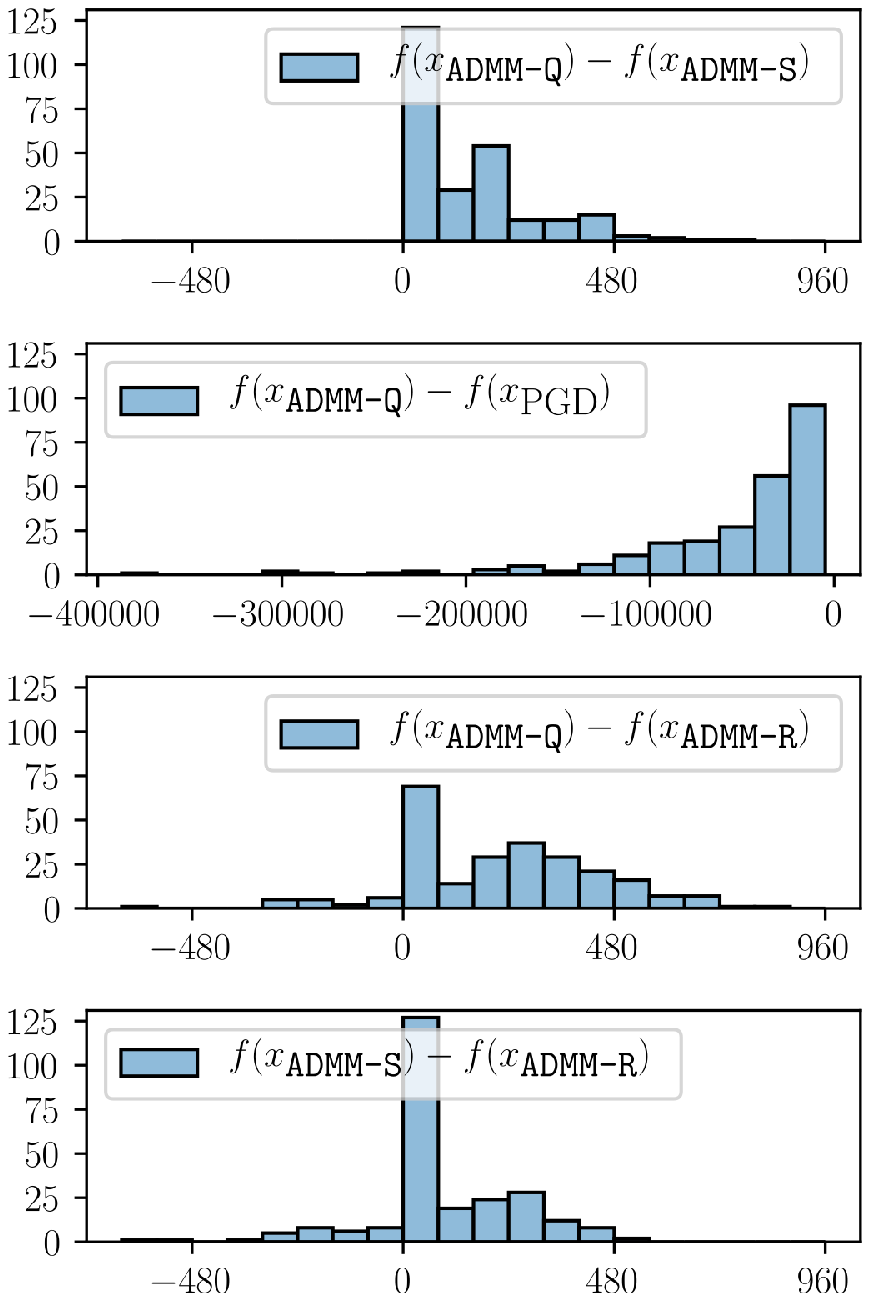}
    }
    \caption{Histogram of the difference of obtained objective values for different algorithm pairs}
    \label{fig:hist_all}
\end{figure}

{\color{black}

\subsection{Neural Network Binarization}
While  ADMM algorithm has been extremely successful in binarization and pruning of neural networks \cite{ye2018progressive, ye2019progressive, yuan2019tp, yuan2019ultra, leng2018extremely, lin2019toward, zhang2018systematic, liu2020autocompress, ren2019admm, li2019admm}, most of these works combine ADMM with other heuristics. To understand the behavior of the ADMM algorithm (independent of  other heuristics), here  we study the performance of pure \admmq~and its variants (with no additional heuristics) when used for binarizing neural networks trained on MNIST and CIFAR-10 datasets.

\subsubsection{MNIST}
\label{subsec : MNIST}

The MNIST dataset~\cite{lecun1998gradient} consists of $28\times 28$ arrays of grayscale pixel images classified into 10 handwritten digits. It includes $60,000$ training images and $10,000$ testing images. The task here is to train a binary-weighted classifier to recognized hand written digits, which can be formulated as
\begin{equation}
    \min_{W}\; \frac{1}{N} \sum_{i=1}^{N}\ell(f(W;x_i),y_i)\;\; \st\;\; W \in \left\{-1,+1\right\}^{d}
\end{equation}
where $(x_i,y_i)$ is the $i$-th training sample; $x_i$ is the input image; $y_i$ is the label; $W$ represents the weights of the network. The work in~\cite{courbariaux2015binaryconnect} used ``Straight Through Estimator'' to binarize the network and reached the accuracy level of the full-precision network. We repeat the experiment with the same network as~\cite{courbariaux2015binaryconnect} and apply \admmq{} and its variants. Similar to the quadratic case, we also compare the performance with PGD and GD+Proj. We conduct two sets of experiments: with pretraining and without pretraining. To the best of our knowledge, all the ADMM-based approaches~\cite{ye2018progressive, ye2019progressive, yuan2019tp, yuan2019ultra, leng2018extremely, lin2019toward, zhang2018systematic, liu2020autocompress, ren2019admm, li2019admm} start from a pre-trained full-precision network.  However, in order to solely study the performance of ADMM-based methods (and not additional modules around it), we avoid using pre-training in some of our experiments. We also did not use any popular heuristics and we relies on implementing our plain ADMM-based algorithms.

To remove the effect of random initialization, we run each algorithm for 5 times and record the mean and standard deviation of the testing accuracy. For algorithms with pre-training, we pre-train the model with full precision and then apply the algorithm. 
Training parameters and network structures be found in Table~\ref{tab:training_para_mnist} and  Table~\ref{tab:net-arch-mnist} in the appendix. Adam optimizer is used for all algorithms. Note that in step~\ref{step:ADMMQYupdate} of algorithm~\ref{alg: ADMM exact}, it is required to solve a minimization problem, which is not always tractable in practice. Thus, here we apply 5 epochs of Adam update on $W$.

\textbf{Results.} Table~\ref{tab:Mnist-simple-network} shows that plain \admmq~and its variants have comparable results with BinaryConnect~\cite{courbariaux2015binaryconnect}. {\color{black} Binarizing the weights saves the storage as much as $96.78\%$ (See Table~\ref{tab:storage-saving}).} One substantial difference between \cite{courbariaux2015binaryconnect} and the proposed work is that we do not use any heuristics and the proposed algorithm enjoys theoretical guarantees. Note that for \admmq{} without pre-training, we fix a value of $\rho$ and keep it until the end of the training process. We observed that one can indeed use ``scheduling'' for  parameter~$\rho$, i.e., increasing it gradually, to shorten the training time. It is worth mentioning that pre-training (with non-binarized weights) in fact further improves the performance of ADMM-based methods (see appendix~\ref{app: simulation_nn_all}). 

\begin{table}[]
\centering
\begin{tabular}{@{}lc@{}}
\toprule
Algorithm      & Accuracy \\ \midrule
{\color{black}BinaryConnect\tablefootnote{BinaryConnect~\cite{courbariaux2015binaryconnect}}}   &     $98.71\%$     \\
Full Precision & $98.87\pm 0.04\%$  \\
GD+Proj   &     $74.92\pm 4.83\%$     \\
\midrule
PGD   & $92.73\pm 0.23\%$         \\ 
\admmq   & $98.21 \pm 0.16\%$         \\ 
\admmr   & \textcolor{black}{$97.78 \pm 0.23\%$}         \\ 
\admms   & $98.21 \pm 0.07\%$         \\ 
\bottomrule
\end{tabular}
\caption{Testing accuracies for MNIST dataset 
}
\label{tab:Mnist-simple-network}
\end{table}

\subsubsection{CIFAR-10}
\label{subsec : CIFAR}

The CIFAR-10 dateset~\cite{krizhevsky2009learning} is a collection of images widely used to train machine learning models. It consists of $32 \times 32$ sized RGB images classified into 10 mutually exclusive categories: airplane, automobile, bird, cat, deer, dog, frog, horse, ship, and truck. The dataset consists of 50,000 training images and 10,000 testing images. Similar to the MNIST experiments, the task is to build a binary-weighted network for classifying the images.

We repeat run our algorithms to train neural networks on CIFAR-10 dataset with and without pretraing. For this experiment, we use  Resnet-18~\cite{he2016deep} architecture. The hyper-parameters used in our experiments are summarized in Table~\ref{tab:training_para_cifar-10}. Adam optimizer is used for all algorithms; We apply 25~epochs of Adam updates on $W$ to solve the minimization problem in step~\ref{step:ADMMQYupdate} of algorithm~\ref{alg: ADMM exact}.

\paragraph{Results.} {\color{black}Table \ref{tab:CIFAR-results} shows the results of the experiments on CIFAR-10. Binarizing the weights saves the storage of up to $96.79\%$ (See Table~\ref{tab:storage-saving}).} Progressive DNN~\cite{ye2019progressive} is the state-of-the-art result which has multiple re-training heuristics involved. One thing worth mentioning here is that we do not use any heuristics. We only use diminishing step-size and increasing rho during the training procedure which is standard. We can see that without pretraining, the results of ADMM-type algorithm are much better than PGD which is consistent with the observations in MNIST and quadratic experiments. \admmr~and \admms~slightly outperform \admmq. Pretraining can further improve the performances of ADMM-type algorithms, making them comparable with the full-precision network (see appendix~\ref{app: simulation_nn_all}).

\begin{table}[]
\centering
\begin{tabular}{@{}lc@{}}
\toprule
Algorithm      & Accuracy \\ \midrule
Progressive DNN\tablefootnote{Progressive DNN~\cite{ye2019progressive}}  &     $93.53\%$     \\
Full Precision & $93.06\%$  \\
GD+Proj   &     $9.86\%$     \\\midrule
PGD   & $63.53\%$         \\ 
\admmq   & $82.74\%$         \\ 
\admmr   & $84.87\%$         \\ 
\admms   & $84.72\%$         \\ 
\bottomrule
\end{tabular}
\caption{Testing accuracies for CIFAR-10 dataset}
\label{tab:CIFAR-results}
\end{table}
}

\begin{table}[]
\color{black}
\centering
\begin{tabular}{@{}lcc@{}}
\toprule
 & Full-precision & Binary \\ \midrule
MNIST & 140.55 MB & 4.53 MB \\
CIFAR-10 & 53.53 MB & 1.72 MB \\ \bottomrule
\end{tabular}
\caption{The storage savings of binarized neural networks}
\label{tab:storage-saving}
\end{table}


\section*{Acknowledgments}
The authors would like to thank Mingyi Hong (University of Minnesota) for the fruitful discussions and invaluable comments that improved the quality of this paper.

\bibliographystyle{abbrvnat}

\bibliography{references.bib}

\clearpage

 \onecolumn

\hsize\textwidth
  \linewidth\hsize \toptitlebar {
  \begin{center}
  \Large\bfseries Supplementary Materials:\\
  Alternating Direction Method of Multipliers for Quantization
\end{center}}
 \bottomtitlebar 



\setcounter{section}{0}
\renewcommand{\thesection}{\Alph{section}}
\section{On the update rules of \admmq}\label{Apx:ExactAdmmDetails}
Consider the following optimization problem mentioned in section~\ref{subsec: admmq_form}:
\[
\min_{x} \quad f(x) + \mathcal{I}_\cA (y) \quad \quad  \st \quad x = y.
\]
Following the steps of regular ADMM in section~\ref{subsec:ADMMReview}, we have:
\[
\mathcal{L}(x,y,\lambda) \triangleq f(x) + \mathcal{I}_\cA(y) + \langle \lambda , x - y\rangle  + \frac{\rho}{2} \|x - y\|_2^2,
\]
In regular ADMM, the order of updating variables does not matter for convergence. When we extend its use to quantization, we update $y$, $x$ and $\lambda$ in sequence at each iteration, which is convenient for analyzing its convergence.
\begin{align}
    & \textrm{Primal Update:} \quad &&y^{r+1}  = \arg\min_{y} \mathcal{L}(x^r,y,\lambda^r), \quad\quad 
    x^{r+1}  = \arg\min_{x} \mathcal{L}(y,x^{r+1},\lambda^r)\nonumber\\
    & \textrm{Dual Update:}\quad &&\lambda^{r+1}  = \lambda^r + \rho ( x^{r+1} - y^{r+1}).  \nonumber
\end{align}
The update rule of $x$ and $\lambda$ is clear. We only derive the update rule of $y$ here:
\begin{equation}
    \begin{aligned}
    y^{r+1}&=\arg\min_{y} \mathcal{L}(x^r,y,\lambda^r)\\
    &=\arg\min_{y}f(x^r)+\mathcal{I}_\cA(y)+\langle \lambda^r , x^r - y\rangle  + \frac{\rho}{2} \|x^r - y\|_2^2\\
    &=\arg\min_{y}\mathcal{I}_\cA(y)+\langle \lambda^r , x^r - y\rangle  + \frac{\rho}{2} \|x^r - y\|_2^2\\
    &=\arg\min_{y}\mathcal{I}_\cA(y)+\langle \lambda^r , x^r - y\rangle  + \frac{\rho}{2} \|x^r - y\|_2^2\\
    &=\arg\min_{y}\mathcal{I}_\cA(y)+\|y-x^r-\rho^{-1}\lambda^r\|^2_2\\
    &=\cP_{\cA}(x^r+\rho^{-1}\lambda^r)
    \end{aligned}
\end{equation}

\section{Proofs in Section~\ref{sec:ADMMQ}}
\label{sec:ProofsofADMMQ}
\begin{lemma} \label{lemma:lambda_r}
For any $r\geq 1$ we have~$\lr = -\nabla_x f(x^r)$.
\end{lemma}
\begin{proof}
based on the algorithm updates and the optimality condition for $\xrn$ we can easily verify that:
\begin{equation}
\nonumber
    \nabla_{x} f(x^{r+1})+\underbrace{\lambda^r+\rho(x^{r+1}-y^{r+1})}_{\lambda^{r+1}} = 0.
\end{equation}
\end{proof}

\noindent\textbf{Lemma~\ref{lemma: lowerbound}.}~
If $\rho\geq L_f$, we have
$ 
\mathcal{L}(\xr,\yr,\lr)\geq f(\yr) \geq  f_{\min},  \quad \forall r\geq 1.
$
\begin{proof}
Note that based on Lemma~\ref{lemma:lambda_r}, we have
\begin{align}
    \nonumber
    \LA(\xr,\yr,\lr) &= f(\xr) + \langle\nabla f(\xr), \yr-\xr \rangle + \frac{\rho}{2}\|\xr-\yr\|^2\\
    &\geq f(\yr) \geq f_{\min}
\end{align}
where the last two inequalities are due to Assumptions \ref{assumption: lipschitz} and \ref{assumption: lowerbdd}, respectively.
\end{proof}

\noindent\textbf{Lemma~\ref{lemma: decrease}.}
Define $\sigma(\rho) \triangleq \rho-\mu$.  We have 
\begin{equation}
    \label{eq: diff}
    \mathcal{L}(x^{r+1},y^{r+1},\lambda^{r+1})-\mathcal{L}(x^r,y^r,\lambda^r)\leq \left(\rho^{-1}{L_f}^2-\frac{\sigma(\rho)}{2}\right)\left\|x^{r+1}-x^r\right\|^2.
\end{equation}

\begin{proof}
Let us re-write \eqref{eq: diff} as
\begin{equation}
\nonumber
\begin{split}
\mathcal{L}(x^{r+1},y^{r+1},\lambda^{r+1})-\mathcal{L}(x^r,y^r,\lambda^r)=\underbrace{\mathcal{L}(x^{r+1},y^{r+1},\lambda^{r+1})-\mathcal{L}(x^{r+1},y^{r+1},\lambda^r)}_{(A)}\\+ \underbrace{\mathcal{L}(x^{r+1},y^{r+1},\lambda^r)-\mathcal{L}(x^r,y^r,\lambda^r)}_{(B)}.
\end{split}
\end{equation}
\\
We want to show that $(A)+(B)\leq0$. First of all note that

\begin{equation}
\nonumber
    (A)=\langle\lambda^{r+1},x^{r+1}-y^{r+1}\rangle-\langle\lambda^r,x^{r+1}-y^{r+1}\rangle=\rho^{-1}\left\|\lambda^{r+1}-\lambda^{r}\right\|^2.
\end{equation}
\\
\noindent By the optimality condition of $x^{r+1}$, we have:
\begin{equation}
\nonumber
    \nabla_{x} f(x^{r+1})+\underbrace{\lambda^r+\rho(x^{r+1}-y^{r+1})}_{\lambda^{r+1}} = 0,
\end{equation}
showing that $\nabla_{x} f(x^{r+1})=-\lambda^{r+1}$, or $\nabla_{x} f(x^{r})=-\lambda^{r}$.

Furthermore, by the lipschitz assumption of $f(\cdot)$, we have $\left\|\nabla_{x}f(x^{r+1})-\nabla_{x}f(x^r)\right\|^2\leq {L_f}^2\left\|x^{r+1}-x^r\right\|^2$, showing that $$\left\|\lambda^{r+1}-\lambda^{r} \right\|^2\leq {L_f}^2\left\|x^{r+1}-x^r\right\|^2.$$\\
Therefore,
\[
(A)\leq\rho^{-1}{L_f}^2\left\|x^{r+1}-x^r\right\|^2.
\]

On the other hand:
\begin{equation}
\nonumber
    \begin{aligned}
        (B)&=\mathcal{L}(x^{r+1},y^{r+1},\lambda^r)-\mathcal{L}(x^r,y^r,\lambda^r)\\
        &=\mathcal{L}(x^{r+1},y^{r+1},\lambda^r)-\mathcal{L}(x^r,y^{r+1},\lambda^r)+\underbrace{\mathcal{L}(x^r,y^{r+1},\lambda^r)-\mathcal{L}(x^r,y^r,\lambda^r)}_{\leq0}\\
        &\leq \mathcal{L}(x^{r+1},y^{r+1},\lambda^r)-\mathcal{L}(x^r,y^{r+1},\lambda^r)\\
        &\leq -\frac{\sigma(\rho)}{2}\left\|x^{r+1}-x^r\right\|^2,
    \end{aligned}
\end{equation}
where $\sigma(\rho)$ is the strong convex modulus of $\mathcal{L}(\cdot, \yrn, \lr)$ (note that $\sigma(\rho) = \rho-\mu$). 
\end{proof}

{
{\color{black}
\begin{example}
\label{example:nonexist}
Consider the optimization problem $\min_{x\in \mathbb{Z}} \frac{1}{2}(x^2-x)$. It is easy to verify that a $\rho$-stationary point does not exist for ${\rho=\frac{1}{2} < L_f =1}$.
\end{example}
}
\color{black}
\noindent\textbf{Lemma~\ref{lemma: optimality}.} Assume $x^\star$ is the global optimal solution to problem~(1), then $x^\star$ is $\rho$-stationary for any $\rho\geq L_f$.
\begin{proof}
Assume the contrary that $x^\star$ is not $\rho$-stationary. By Definition~\ref{def: stationarity}, $x^\star\notin \arg\min_{a\in\mathcal{A}} \frac{\rho}{2}\|a-x^\star+\rho^{-1}\nabla f(x^\star)\|^2$. Expanding the objective and adding the constant term $f(x^\star)-\frac{1}{2\rho}\|\nabla f(x^\star)\|^2$ implies that
\begin{equation}
\nonumber
    x^\star\notin \arg\min_{a\in\mathcal{A}}~\left(\widehat{f}(a; x^\star):= f(x^\star)+\langle \nabla f(x^\star),\;a-x^\star\rangle + \frac{\rho}{2}\|a-x^\star\|^2\right).
\end{equation}
Since $\rho\geq L_f$, we have $f(a)\leq \widehat{f}(a; x^\star)$, $\forall~a$, according to the descent lemma. Moreover, there exists ${a^\star \in \arg\min_{a\in\mathcal{A}}\widehat{f}(a;x^\star)}$ by the compactness of $\mathcal{A}$. Thus, $f(a^\star)\leq\widehat{f}(a^\star;x^\star)<\widehat{f}(x^\star;x^\star)=f(x^\star)$, which contradicts the optimality of $x^\star$. 
\end{proof}
}
\noindent\textbf{Theorem~\ref{Thm:ADMMQConvergence}.}
Assume $(\xb, \yb, \lb)$ is a limit point of the \admmq~algorithm. Then $\xb$ is a $\rho$--stationary point of the optimization problem~\eqref{eq:original}.
\begin{proof}
Consider a sub-sequence $(\xrt,\yrt,\lrt)$, for $t=0,\cdots$ which converges to $(\xb, \yb, \lb)$. First of all due to decrease lemma \ref{lemma: decrease} and lower boundedness of augmented Lagrangian, Lemma \ref{lemma: lowerbound}, we know that ${\lim_{t\rightarrow \infty}\|x^{r_t+1}-\xrt\|=0}$. Thus, 
\begin{equation}
    \lim_{t\rightarrow \infty} x^{r_t+1} = \xb
\end{equation}
Now based on Lemma~\eqref{lemma:lambda_r}, we also know that
\begin{align}
    \lb = \lim_{t\rightarrow \infty} \lrt = \lim_{t\rightarrow \infty} \nabla f(\xrt) = \nabla f(\xb)\\
    \lim_{t\rightarrow \infty} \lambda^{r_t+1} =  \lim_{t\rightarrow \infty} \nabla f(x^{r_t+1}) = \nabla f(\xb) 
\end{align}
Thus, $\lim_{t\rightarrow \infty} \lambda^{r_t+1} = \lb$. 

Also, as $\mathcal{A}$ is finite, there exists a large enough $T$, such that $\yrt = \bar{y}$ for $t\geq T$. Again due to the fact that $\mathcal{A}$ is finite, we can re-fine the sub-sequence such that $y^{r_t+1} = \hat{y}$. Thus, without loss of generality assume that these two conditions hold, i.e. $\yrt = \bar{y}$ and $y^{r_t+1} = \hat{y}$ for all $t$ for an appropriately refined sub-sequence. This means that
\begin{equation}
\label{eq: y_hat1}
    \hat{y} \in \arg\min_{a}\|a-(\xrt+\rho^{-1}\lrt)\|
\end{equation}
Moreover, $\lambda^{r_t+1} = \lrt + \rho(\hat{y}-\xrt)$. Taking the $\lim_{t\rightarrow\infty}$ from both sides, we get
\begin{equation}
    \hat{y} = \xb.
\end{equation}
Combining the above with \eqref{eq: y_hat1} we can easily see that
\begin{equation}
    \|\xb-(\xrt+\rho^{-1}\lrt)\| \leq \|a_i-(\xrt+\rho^{-1}\lrt)\|,~i=0,\cdots,N
\end{equation}
Taking the limits $\lim_{t\rightarrow\infty}$ from both hand sides of the inequality for all the points $a_i$ we have
\begin{equation}
    \|\xb-(\xb+\rho^{-1}\lb)\| \leq \|a_i-(\xb+\rho^{-1}\lb)\|,~i=0,\cdots,N.
\end{equation}
Thus, 
\begin{equation}
    \xb\in \arg\min_{a\in\mathcal{A}}\|a-(\xb-\rho^{-1}\nabla f(\xb))\|,
\end{equation}
where we used the fact that $\lb = -\nabla f(\xb)$.
\end{proof}

\section{Convergence Analysis for \iqadmm}
\label{app: proof_inexact}
In order to prove the main convergence results, we need a few definitions and helper lemmas. Throughout this section we re-state all the theoretical results and prove them in the order we need them. For a reference of the steps in the algorithm see Algorithm~\ref{alg: inexact_admm}.

First, let us define:
\begin{align}
    e^r = \nabla_x\mathcal{L}(x^r,y^{r},\lambda^{r-1}) = \nabla f(x^r) + \lambda^{r-1} + \rho(x^r-y^{r}) = \nabla f(x^r) + \lambda^{r}
\end{align}


\begin{lemma} 
\label{lemma: strong-convexity}
Due to $\sigma(\rho)$-strong convexity and $(L_f+\rho)$-smoothness of $\mathcal{L}(\cdot, y^{r}, \lambda^{r-1})$, we know that
\begin{align}
    \sigma(\rho) \|x^r-x_\star^r\|\leq \|e^r\| \leq (\rho + L_f)\|x^r-x_\star^r\|
\end{align}
Moreover, due to strong convexity we also know that:
\begin{align}
    \langle e^r, x^r-x_\star^r\rangle \geq \sigma(\rho) \|x^r-x_\star^r\|^2
\end{align}
\end{lemma}

\begin{lemma} 
\label{lemma: lower_bound_inexact}
If $\rho\geq L_f$ and we also assume that the iterates $x^r$ stay bounded. Then there exists a non-negative number $\bar{D}$ s.t. $\|x^r-y^r\|\leq \bar{D}$. With this definition, 
\begin{align}
    \mathcal{L}(x^r, y^r, \lambda^r) \geq f_{\min} - {\gamma}(\rho+L_f)\bar{D}^2
\end{align}
\end{lemma}
\begin{proof}
    Note that
    \begin{align}
      \mathcal{L}(x^r, y^r, \lambda^r) &=
      f(x^r) + \langle \lambda^r, x^r-y^r\rangle + \frac{\rho}{2}\|x^r-y^r\|^2\\
      & = \underbrace{f(x^r) + \langle\nabla f(x^r), y^r-x^r \rangle + \frac{\rho}{2}\|x^r-y^r\|^2}_{\geq f(y^r)}+ \langle e^r, x^r-y^r\rangle\\
      &\geq f(y^r) - \|e^r\|\|x^r-y^r\|\\
      &\geq f_{\min} - {\gamma} (\rho+L_f)\bar{D}^2
    \end{align}
    where the last inequality is due to the assumptions and Lemma \ref{lemma: strong-convexity}.
\end{proof}
Now let us prove sufficient decrease on $\mathcal{L}$ in each iteration.
\begin{lemma}
\label{lemma: decrease_inexact}
Let the assumptions of Lemma \ref{lemma: lower_bound_inexact} be true. Also, define 
\begin{equation}
    \alpha = \bigg(\frac{2L_f^2}{\rho} + \frac{4(\rho+L_f)^2{\gamma}^2}{\rho}+ \frac{{\gamma}^2(\rho+L_f)}{2}-\frac{(1-{\gamma})^2\sigma(\rho)}{2}\bigg)
\end{equation}
and $\beta = \frac{4(\rho+L_f)^2{\gamma}^2}{\rho}$. Note that $\sigma(\rho) = \rho-\mu\geq 0$. Furthermore, assume that the parameters $\rho$ and ${\gamma}$ are chosen such that $\alpha+\beta<0$. Then, have
\begin{equation}
    \lim_{r\rightarrow\infty}\|x^{r+1}-x^r\| = 0.
\end{equation}

\end{lemma}

\begin{proof}
Let us re-write \eqref{eq: diff} as
\begin{equation}
\nonumber
\begin{split}
\mathcal{L}(x^{r+1},y^{r+1},\lambda^{r+1})-\mathcal{L}(x^r,y^r,\lambda^r)=\underbrace{\mathcal{L}(x^{r+1},y^{r+1},\lambda^{r+1})-\mathcal{L}(x^{r+1},y^{r+1},\lambda^r)}_{(A)}\\+ \underbrace{\mathcal{L}(x^{r+1},y^{r+1},\lambda^r)-\mathcal{L}(x^r,y^r,\lambda^r)}_{(B)}.
\end{split}
\end{equation}
\\
We want to show that $(A)+(B)\leq0$.

\begin{equation}
\nonumber
    (A)=\langle\lambda^{r+1},x^{r+1}-y^{r+1}\rangle-\langle\lambda^r,x^{r+1}-y^{r+1}\rangle=\rho^{-1}\left\|\lambda^{r+1}-\lambda^{r}\right\|^2.
\end{equation}
Using our definitions, we have
\begin{align}
    (A) &= \rho^{-1}\|\lambda^{r+1}-\lambda^r\|^2\\
    & = \rho^{-1} \|\nabla f(x^{r+1})-\nabla f(x^{r})+e^r-e^{r+1}\|^2\\
    & \leq \frac{2}{\rho}\bigg(\|\nabla f(x^{r+1})-\nabla f(x^r)\|^2 + \|e^{r+1}-e^r\|^2\bigg)\\
    &\leq \frac{2}{\rho}\bigg(L_f^2\|x^{r+1}-x^r\|^2+2\|e^r\|^2 + 2\|e^{r+1}\|^2\bigg)\\
    &\leq \frac{2}{\rho}\bigg(L_f^2\|x^{r+1}-x^r\|^2+2(\rho+L_f)^2{\gamma}^2\bigg(\|x^{r+1}-x^r\|^2 + \|x^r-x^{r-1}\|^2\bigg)\bigg), 
\end{align}
where the last inequality is due to Lemma \ref{lemma: strong-convexity} and the way $x^r$ is chosen in Algorithm \ref{alg: inexact_admm}.

On the other hand:
\begin{equation}
\nonumber
    \begin{aligned}
        (B)&=\mathcal{L}(x^{r+1},y^{r+1},\lambda^r)-\mathcal{L}(x^r,y^r,\lambda^r)\\
        &=\mathcal{L}(x^{r+1},y^{r+1},\lambda^r)-\mathcal{L}(x^r,y^{r+1},\lambda^r)+\underbrace{\mathcal{L}(x^r,y^{r+1},\lambda^r)-\mathcal{L}(x^r,y^r,\lambda^r)}_{\leq 0 \text{ (due to update of $y$)}}\\
        &\leq \mathcal{L}(x^{r+1},y^{r+1},\lambda^r)-\mathcal{L}(x^r,y^{r+1},\lambda^r)\\
        & = \underbrace{\mathcal{L}(x^{r+1},y^{r+1},\lambda^r) - \mathcal{L}(x_\star^{r+1},y^{r+1},\lambda^r)}_{\leq\frac{L_f+\rho}{2}\|x^{r+1}-x_\star^{r+1}\|^2 } + \underbrace{\mathcal{L}(x_\star^{r+1},y^{r+1},\lambda^r)- \mathcal{L}(x^r,y^{r+1},\lambda^r)}_{\leq -\frac{\sigma(\rho)}{2}\|x_\star^{r+1}-x^r\|^2}\\
        &\leq \frac{L_f+\rho}{2}\|x^{r+1}-x_\star^{r+1}\|^2 -\frac{\sigma(\rho)}{2}\|x_\star^{r+1}-x^r\|^2 ,
    \end{aligned}
\end{equation}
Now note that $\|x^{r}-x_\star^{r+1}\|\geq (1-{\gamma})\|x^{r+1}-x^r\|$ and $\|x^{r+1}-x_\star^{r+1}\|\leq {\gamma} \|x^{r+1}-x^r\|$ because of the update rules of Algorithm \ref{alg: inexact_admm}. Plugging in these, we get
\begin{align}
    (B) \leq \bigg(\frac{{\gamma}^2 (\rho+L_f)}{2}-\frac{(1-{\gamma})^2\sigma(\rho)}{2}\bigg)\|x^{r+1}-x^r\|^2
\end{align}
Now combining the inequalities for $(A)$ and $(B)$, we have
\begin{align}
    &\mathcal{L}(x^{r+1},y^{r+1},\lambda^{r+1})-\mathcal{L}(x^r,y^r,\lambda^r) \\
    &\leq \underbrace{\bigg(\frac{2L_f^2}{\rho} + \frac{4(\rho+L_f)^2{\gamma}^2}{\rho}+ \frac{{\gamma}^2(\rho+L_f)}{2}-\frac{(1-{\gamma})^2\sigma(\rho)}{2}\bigg)}_{\alpha}\|x^{r+1}-x^r\|^2 + \underbrace{\frac{4(\rho+L_f)^2{\gamma}^2}{\rho}}_{\beta}\|x^r-x^{r-1}\|^2
\end{align}
Now for any $T$:
\begin{align}
    f_{\min} - {\gamma}(\rho+L_f)\bar{D}^2 
    &\leq\mathcal{L}(x^{T+1}, y^{T+1}, \lambda^{T+1})\\
    &= \mathcal{L}(x^0, y^0, \lambda^0)+ \sum_{r=0}^T \mathcal{L}(x^{r+1},y^{r+1},\lambda^{r+1})-\mathcal{L}(x^r,y^r,\lambda^r)\\
    &\leq (\alpha+\beta)\sum_{r=0}^{T-1}\|x^{r+1}-x^{r}\|^2 ~+~ \alpha \|x^{T+1}-x^T\|^2 + \mathcal{L}(x^0, y^0, \lambda^0)\\
    &\leq (\alpha+\beta) \sum_{r=0}^{T}\|x^{r+1}-x^{r}\|^2 + \mathcal{L}(x^0, y^0, \lambda^0),
\end{align}
where the last inequality is due to the fact the $\beta\geq0$. Now if the parameters are chosen appropriately such that $\alpha+\beta<0$, then the right hand side of the above inequality is decreasing as $T$ increases, while the left hand side is constant. Therefore, we have $\lim_{T\rightarrow\infty}\sum_{r=0}^T\|x^{r+1}-x^r\|^2<\infty$. Thus, $\lim_{r\rightarrow\infty}\|x^{r+1}-x^r\| = 0$.
\end{proof}

\begin{theorem}
Assume that all the assumptions of Lemma \ref{lemma: decrease_inexact} is satisfied. Then, For any limit point $(\xb, \yb, \lb)$ of the Algorithm \ref{alg: inexact_admm}, $\xb$ is a stationary solution of the problem.
\end{theorem}
\begin{proof}

Consider a sub-sequence $(\xrt,\yrt,\lrt)$, for $t=0,\cdots$ which converges to $(\xb, \yb, \lb)$. First of all due to Lemma \ref{lemma: decrease_inexact}, we know that $\lim_{t\rightarrow \infty}\|x^{r_t+1}-\xrt\|=0$ and $\lim_{t\rightarrow\infty}\|x^{r_t-1}-\xrt\|=0$. Thus, 
\begin{equation}
    \lim_{t\rightarrow \infty} x^{r_t+1} = \xb ~~\&~~\lim_{t\rightarrow \infty} x^{r_t-1} = \xb 
\end{equation}
Moreover, due to the updates of the algorithm
\begin{equation}
    \lim_{t\rightarrow\infty} \|x^{r_t+1}-x_\star^{r_t+1}\| \leq \lim_{t\rightarrow\infty} {\gamma} \|x^{r_t+1}-x^{r_t}\| = 0 ~~\&~~\lim_{t\rightarrow\infty} \|x^{r_t}-x_\star^{r_t}\| \leq \lim_{t\rightarrow\infty} {\gamma} \|x^{r_t}-x^{r_t-1}\| = 0
\end{equation}
Thus, $\lim_{t\rightarrow\infty} e^{r_t} = \lim_{t\rightarrow\infty} e^{r_t+1} = 0$, which means
\begin{align}
    \lb = \lim_{t\rightarrow \infty} \lrt = - \lim_{t\rightarrow \infty} (\nabla f(\xrt)-e^{r_t}) = -\nabla f(\xb)\\
    \lim_{t\rightarrow \infty} \lambda^{r_t+1} =  -\lim_{t\rightarrow \infty} (\nabla f(x^{r_t+1})-e^{r_{t+1}}) = -\nabla f(\xb) 
\end{align}
Thus, $\lim_{t\rightarrow \infty} \lambda^{r_t+1} = \lb$. 

Also, as $\mathcal{A}$ is finite, there exists a large enough T, such that $\yrt = \bar{y}$ for $t\geq T$. Again due to the fact that $\mathcal{A}$ is finite, we can re-fine the sub-sequence such that $y^{r_t+1} = \hat{y}$. Thus, without loss of generality assume that these two conditions hold, i.e. $\yrt = \bar{y}$ and $y^{r_t+1} = \hat{y}$ for all $t$ for an appropriately refined sub-sequence. This means that
\begin{equation}
\label{eq: y_hat}
    \hat{y} \in \arg\min_{a}\|a-(\xrt+\rho^{-1}\lrt)\|
\end{equation}
Moreover, $\lambda^{r_t+1} = \lrt + \rho(x^{r_t+1}-\hat{y})$. Taking the $\lim_{t\rightarrow\infty}$ from both sides, we get
\begin{equation}
    \hat{y} = \xb.
\end{equation}
Combining the above with \eqref{eq: y_hat} we can easily see that
\begin{equation}
    \|\xb-(\xrt+\rho^{-1}\lrt)\| \leq \|a_i-(\xrt+\rho^{-1}\lrt)\|,~i=0,\cdots,N
\end{equation}
Taking the limits $\lim_{t\rightarrow\infty}$ from both hand sides of the inequality for all the points $a_i$ we have
\begin{equation}
    \|\xb-(\xb+\rho^{-1}\lb)\| \leq \|a_i-(\xb+\rho^{-1}\lb)\|,~i=0,\cdots,N.
\end{equation}
Thus, 
\begin{equation}
    \xb\in \arg\min_{a\in\mathcal{A}}\|a-(\xb-\rho^{-1}\nabla f(\xb))\|,
\end{equation}
where we used the fact that $\lb = -\nabla f(\xb)$.
\end{proof}
{
\color{black}
\section{Convergence Analysis of PGD Algorithm}
\label{sec:PGDConvergence}
In this short section, we show that the convergence behavior of Projected Gradient Descent (PGD) algorithm can also be analyzed using Definition \ref{def: stationarity}. Each iteration of PGD  is gradient descent  followed by a projection to the discrete set~$\cA$. More precisely, PGD update rule is given by 
\begin{equation} \label{eq:PGD-Iterate}
    x^{r+1}\in\mathcal{P}_\mathcal{A}(x^r-\rho^{-1}\nabla_x f(x^{r}))
\end{equation}
\begin{lemma}
\label{lemma: PGD-descent}
Consider the PGD algorithm with the update rule $x^{r+1} \in \mathcal{P}_\mathcal{A}(x^r-\rho^{-1}\nabla_x f(x^{r}))$ with $\rho \geq L_f$. Then, for any $r \geq 1$ we have $f(x^r) \geq f(x^{r+1}) \geq f_{min}$.
\end{lemma} 
\begin{proof}
By the update rule of PGD algorithm, we have:
\begin{align*}
    x^{r+1}&\in \arg\min_{a\in\cA}\|a-x^r+\rho^{-1}\nabla f(x^r)\|^2 \\
    &\in \arg\min_{a\in\cA} f(a;x^r):= f(x^r) + \langle \nabla f(x^r),\; a-x^r \rangle +\frac{\rho}{2}\|a-x^r\|^2.
\end{align*}
Since $\rho \geq L_f$, we have $f(a;x^r)\geq f(a),\; \forall a$. Hence $f(x^{r})=f(x^{r};x^{r}) \geq f(x^{r+1};x^{r})\geq f(x^{r+1})$.
\end{proof}
\begin{theorem}
Assume that $f$ satisfies  Assumptions~\ref{assumption: lowerbdd}, \ref{assumption: lipschitz} and \ref{assumption: hessianbddbelow}. Assume further that $\rho$ is chosen large enough so that $\rho \geq L_f$. Let $\xb$ be a limit point of the PGD~algorithm. Then $\xb$ is a $\rho$--stationary point of the optimization problem~\eqref{eq:original}.
\end{theorem}
\begin{proof}
By Lemma~\ref{lemma: PGD-descent} and compactness of $\mathcal{A}$, we know the sequence $f(x^r)$ is bounded and monotone, and hence convergent, i.e. $\lim_{r\to \infty}f(x^r)=\widebar{f}$. On the other hand, the continuity of $f(\cdot)$ implies that:
$$\exists \{x^{r_t}\}\quad\text{s.t.}\quad\lim_{t\to\infty} x^{r_t}=\widebar{x}\in\cA,\;\lim_{t\to\infty} f(x^{r_t})=f(\widebar{x}).$$
Hence, $\lim_{r\to\infty}f(x^r)=f(\widebar{x})$.
Moreover, for any fixed $a\in\cA$, we have 
$$ f(x^{r_t+1})\leq f(x^{r_t}) + \langle \nabla f(x^r_t),\;a-x^r \rangle +\frac{\rho}{2}\|a-x^{r_t}\|^2.$$
Letting $t\to\infty$, we obtain:
$$ f(\widebar{x})\leq f(\widebar{x}) + \langle \nabla f(\widebar{x}),\; a-\widebar{x} \rangle +\frac{\rho}{2}\|a-\widebar{x}\|^2,$$
which in turn implies that:
\begin{equation*}
    \widebar{x} \in \arg\min_{a\in\cA} f(\widebar{x}) + \langle \nabla f(\widebar{x}),\; a-\widebar{x} \rangle +\frac{\rho}{2}\|a-\widebar{x}\|^2
\end{equation*}
or equivalently,
\begin{equation*}
    \widebar{x} \in \arg\min_{a\in\cA}\|a-\widebar{x}+\rho^{-1}\nabla f(\widebar{x})\|^2.
\end{equation*}
Hence, $\widebar{x}$ is a $\rho$-stationary point.
\end{proof}

}
\section{On the update rules of \admms{}  and its behavior}
\label{app: derivation_algo_4}

The update rules of $x$ and $\lambda$ variables are similar to the \admmq~algorithm. Here we only present the $y$~update rule.
Let us define  $\beta'=\beta\rho^{-1}$, $z^{r+1}=x^r+\rho^{-1}\lambda^r$ and $\widetilde{z}^{r+1}=\mathcal{P}_\cA(z^{r+1})$. Following the steps of regular ADMM, the update rule of $y$ can be written as
\begin{equation}
    \begin{aligned}
    y^{r+1}&=\arg\min_{y}\;  \mathcal{L} (x^r,y,\lambda^r) \\
    &=\arg\min_{y}\; f(x^r) + \langle \lambda^r, x^r- y\rangle + \frac{\rho}{2} \|x^r-y\|_2^2 + \beta \mathcal{S}_{\cA} (y) \\
    &=\arg\min_{y}\; \frac{1}{2} \|y-x^r-\rho^{-1}\lambda^r\|^2+\beta\rho^{-1}\mathcal{S}_\cA(y)\\
    &=\arg\min_{y}\; \frac{1}{2} \|y-z^{r+1}\|^2+\beta'\|y-\mathcal{P}_\cA(z^{r+1}) \|_2\\
    &=\arg\min_{y}\; \frac{1}{2} \|y-z^{r+1}\|^2+\beta'\|y-\widetilde{z}^{r+1} \|_2
    \end{aligned}
\end{equation}
If $y^{r+1}\neq\widetilde{z}^{r+1}$, then we can take the derivative of the above function and set it to $0$ to get the update rule of $y$:
\begin{equation}
\label{eq:notEqualADMM-SUpdate}
\begin{aligned}
&(y^{r+1}-z^{r+1})+\beta'\frac{y^{r+1}-\widetilde{z}^{r+1}}{\|y^{r+1}-\widetilde{z}^{r+1}\|}_2=0\\
&\Longrightarrow y^{r+1}=z^{r+1}+\beta'\frac{\widetilde{z}^{r+1}-z^{r+1}}{\|\widetilde{z}^{r+1}-z^{r+1}\|_2}
\end{aligned}
\end{equation}
Now we need to find out when the solution is $y^{r+1}= \widetilde{z}^{r+1}$ and when it is given by equation~\eqref{eq:notEqualADMM-SUpdate}. Using the sub-gradient of the function~$\|y-\widetilde{z}^{r+1}\|_2$ at the point~$y= \widetilde{z}^{r+1}$, we obtain that 
\[
y^{r+1} = \widetilde{z}^{r+1} \quad \textrm{if}\quad \|\widetilde{z}^{r+1} - z^{r+1}\|_2 \leq \beta^\prime
\]
Combining this equation with~\eqref{eq:notEqualADMM-SUpdate}, we obtain the following update rule for $y$:
\[
y^{r+1}=\left\{
		\begin{array}{cl}
		z^{r+1} + \dfrac{\beta'(\widetilde{z}^{r+1}-z^{r+1})}{\|\widetilde{z}^{r+1}-z^{r+1}\|_2} &  ,\;\beta'\leq \|\widetilde{z}^{r+1}-z^{r+1}\|_2  \\
		\widetilde{z}^{r+1} & ,\;\beta' >  \|\widetilde{z}^{r+1}-z^{r+1}\|_2
		\end{array}
		\right.
\]

Notice that this update rule would keep $y^{r+1}$ very close to the set $\mathcal{A}$, especially when $\beta$ is large. In fact in the extreme case where $\beta$ is large enough, i.e. when  $\beta^\prime = \frac{\beta}{\rho}\geq \sup_{z} \| z - \cP_\cA(z)\|$, the update rule of $y$ in \admms~coincide with the update rule of $y$ in \admmq~algorithm. Obviously due to the fact that $y^{r}$ is not in $\mathcal{A}$, we cannot expect the \admms~to converge to a stationary solution defined in Definition \ref{def: stationarity}. But in what follows we show that under assumptions similar to what we used for \admmq~, we can actually show that the Lagrangian function converges in \admms.

Most of the proofs follow the same steps as in the convergence analysis of \admmq. Thus, they are mostly omitted and we only focus on the overall steps and the results here.
First of all it is easy to verify that the result of Lemma~\ref{lemma:lambda_r} is also true for~\admms, i.e. $\lr = -\nabla_x f(x^r)$. 
Moreover, Let us assume that the $y^r$ iterates stay bounded, i.e. $y^r\in \mathcal{A}^\prime$, where $\mathcal{A}^\prime$ is a compact set. Note that this is a reasonable assumption due to the proximity of $y^r$ to the bounded set $\mathcal{A}$. As $f$ is continuous, we can assume there exists a $f_{\min}$ such that $f(y)\geq f_{\min}$ for all $y\in\mathcal{A}^\prime$. Under these assumptions we have the following lemma, which states that the Lagrangian function is lower bounded.

\begin{lemma}
\label{lemma: lowerbound_admms}
If $\rho\geq L_f$, we have
$ 
\mathcal{L}(\xr,\yr,\lr)\geq f(\yr) \geq  f_{\min},  \quad \forall r\geq 1.$
\end{lemma}
The proof is similar to the proof of Lemma~\ref{lemma: lowerbound} and is omitted. 
Moreover, we have the following result which is similar to Lemma~\ref{lemma: decrease} for \admmq.

\begin{lemma}\label{lemma: decrease_admms}
Define $\sigma(\rho) \triangleq \rho-\mu$.  We have 
\begin{equation}
    \mathcal{L}(x^{r+1},y^{r+1},\lambda^{r+1})-\mathcal{L}(x^r,y^r,\lambda^r)\leq (\rho^{-1}L_f^2-\frac{\sigma(\rho)}{2})\left\|x^{r+1}-x^r\right\|^2.
\end{equation}
\end{lemma}
The proof of this lemma also follows the same arguments provided in the proof of Lemma~\ref{lemma: decrease}. 
Based on these two lemmas, we have that augmented Lagrangian function is decreasing and lower bounded when $\rho$ is chosen appropriately. Thus, it has to converge:
\begin{proposition}
If $\rho$ is chosen such that $\rho^{-1}L_f^2-\frac{\sigma(\rho)}{2}$<0, then $\mathcal{L}(x^r,y^r,\lambda^r)$ is decreasing and lower bounded. Thus, it converges.
\end{proposition}

\section{Simulations on Convex Quadratic Case}
\label{app: simulation_quad_all}
Recall in section~\ref{sec: numerical_experiments}, we solve the following problem:
\begin{equation}
    \min_{x} \;\; \frac{1}{2}x^\top Q x + b^\top x \quad \st \;\; x\in \cA \triangleq v\mathbb{Z}^d,
\end{equation}
for some given $Q\in \mathbb{R}^{d\times d}$, $b\in \mathbb{R}^{d}$, and $v\in \mathbb{Z}^+$. We generate matrix $Q$ via the rule~$Q = \widetilde{Q}^\top \widetilde{Q} + \widetilde{q} \widetilde{q}^\top$, where $\widetilde{Q}_{ij}\sim N(0, 1),\;\widetilde{q}_{i}\sim N(0,\sigma_{\widetilde{q}}^2),\;1\leq i, j\leq d$. We follow the same procedure as discussed in section~\ref{sec: numerical_experiments}; see Table~\ref{tab: hyper_parameters} for the hyper-parameters used in \admmq, \admms~and \admmr. We report the results for the following combinations of $v$, $d$ and $\sigma_{\widetilde{q}}^2$ as seen in Table~\ref{tab:parameters_pairs}.

\textbf{Results.} Most of the observations  in section~\ref{sec: numerical_experiments} carry over here regardless of the values of $d$ and $\sigma_{\widetilde{q}}^2$. 
More precisely, \admmq{} outperforms PGD and GD+Proj with   large margins. Both \admms{} and \admmr{}  not only have better median final objective values, but also smaller variance as compared with \admmq. More importantly, the median tends to overlap with the $25\%$ quantile, see Figure~\ref{fig:Quantiles_d16_10}. It means the objective of at least 25 runs are exactly the same as the minimal objective over 50 runs. We also observe that \admms~or \admmr~is not always better than \admmq. As we conduct more experiments, we observed cases that \admms~yields large objective value; see, e.g., instance~3 in Figure~\ref{fig:Quantiles_d16_30}, and compare with Figure~\ref{fig:Quantiles}. Having said that, we observe that \admms~and \admmr~outperform \admmq~in most instances.
\vfill

\begin{table}
\centering
\begin{tabular}{ccc}
\hline
\multicolumn{3}{c}{Parameter Pairs} \\ \hline
$v$ & $d$ & $\sigma_{\widetilde{q}}^2$ \\
8 & 8 & 30 \\
8 & 16 & 30 \\
8 & 32 & 30 \\
8 & 64 & 30 \\
8 & 16 & 10 \\
8 & 16 & 50 \\
8 & 16 & 70 \\ \hline
\end{tabular}
\caption{Parameter pairs used in the experiment}
\label{tab:parameters_pairs}
\end{table}

\begin{figure}[H]
    \centering
    \resizebox{\textwidth}{!}{%
    \includegraphics[]{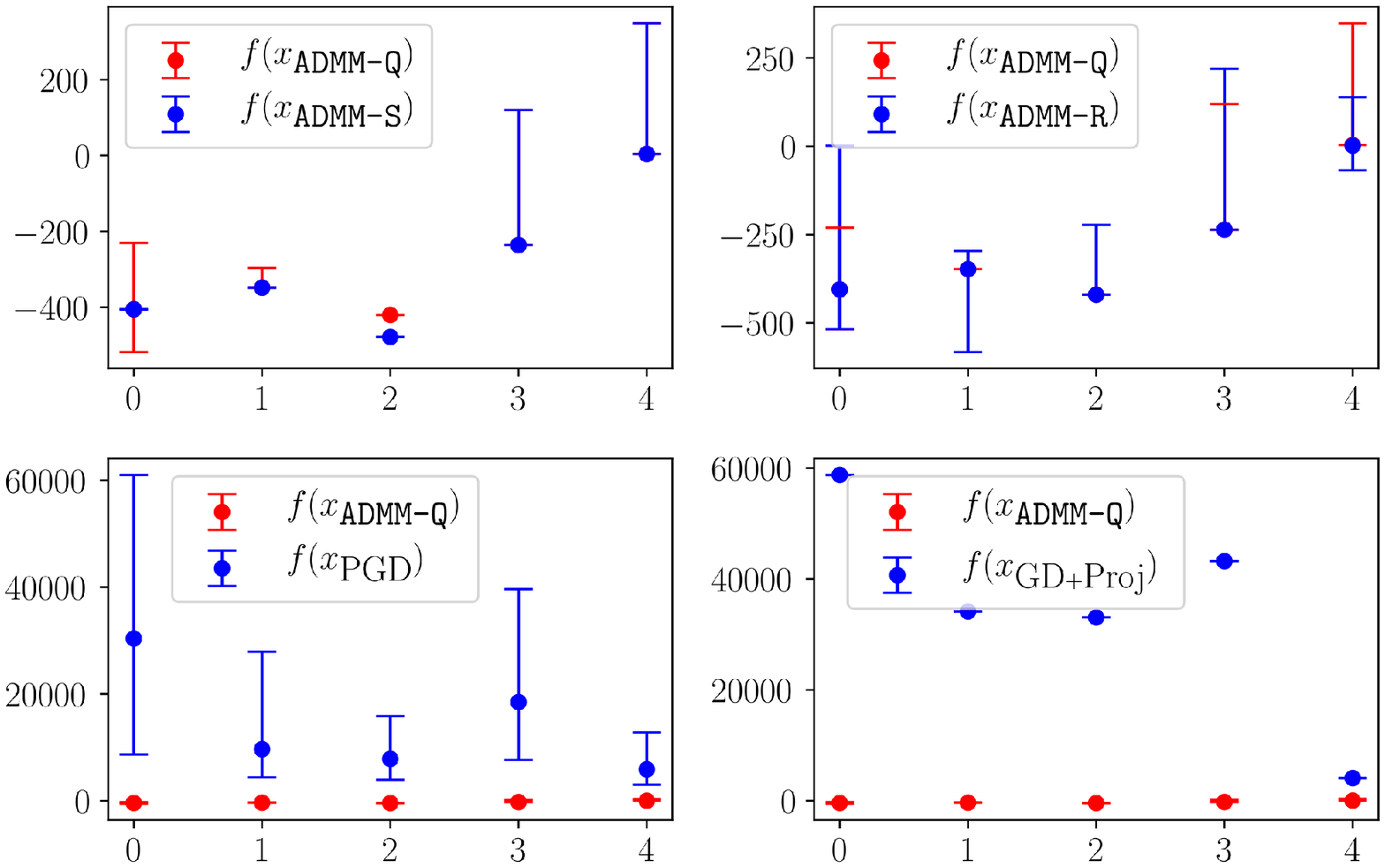}
    }
    \caption{Performance of \admmq, \admms, \admmr~ and PGD on different problem instances with $d=8$, $\sigma_{\widetilde{q}}^2=30$}
    \label{fig:Quantiles_d8_30}
\end{figure}

\begin{figure}[H]
    \centering
    \resizebox{\textwidth}{!}{%
    \includegraphics[]{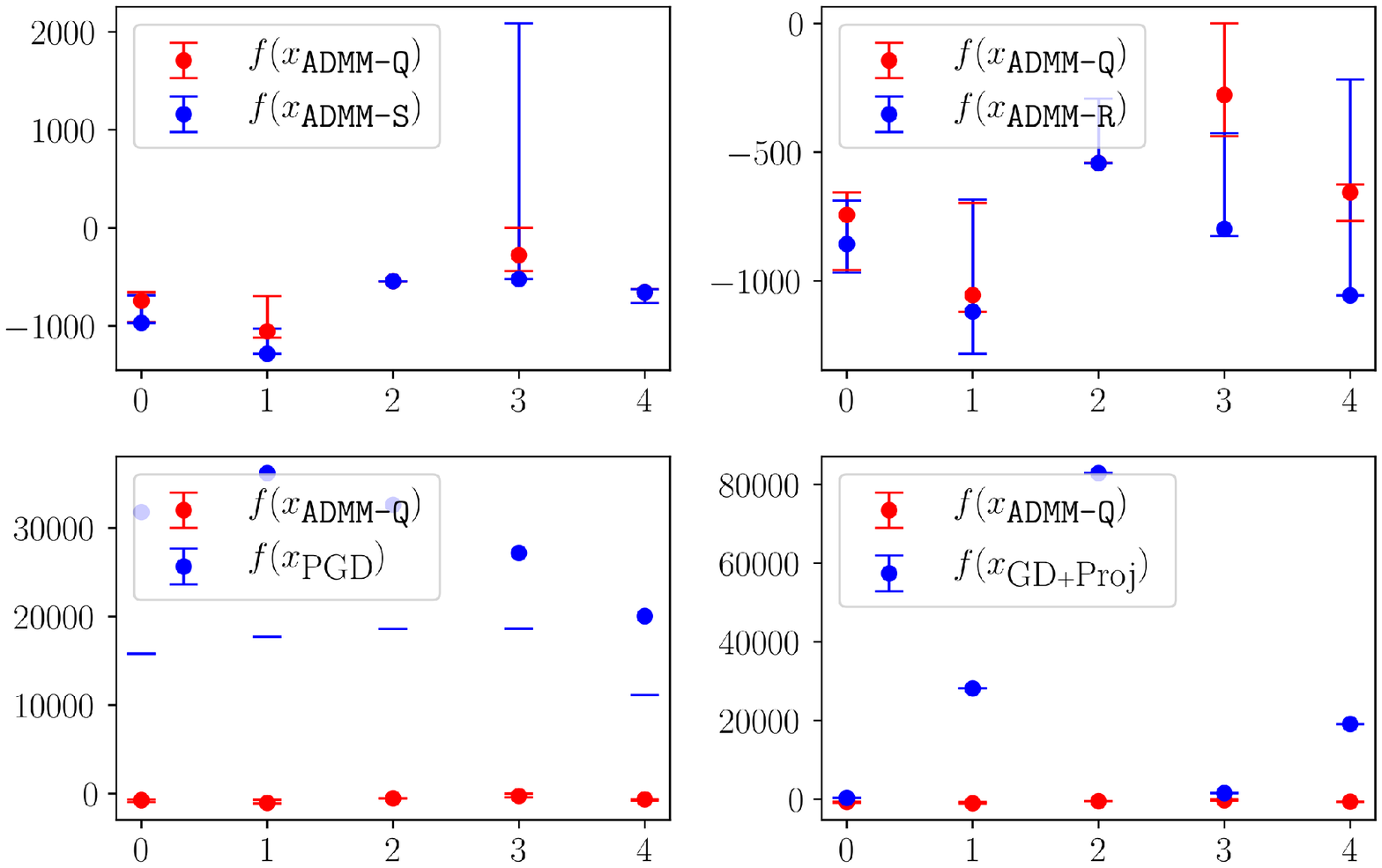}
    }
    \caption{Performance of \admmq, \admms, \admmr~ and PGD on different problem instances with $d=16$, $\sigma_{\widetilde{q}}^2=30$, note the difference compared with Figure~\ref{fig:Quantiles}}
    \label{fig:Quantiles_d16_30}
\end{figure}

\begin{figure}[H]
    \centering
    \resizebox{\textwidth}{!}{%
    \includegraphics[]{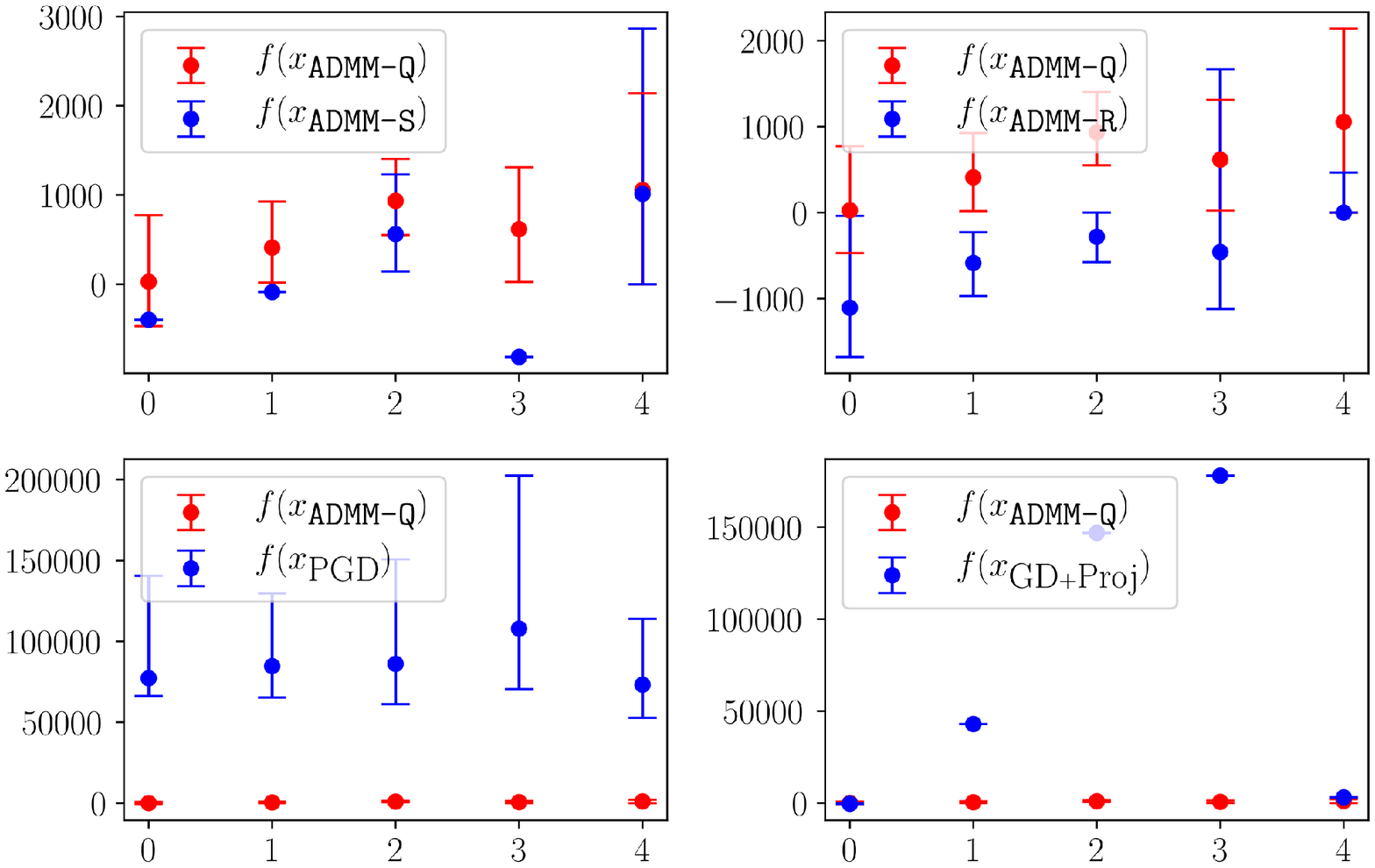}
    }
    \caption{Performance of \admmq, \admms, \admmr~ and PGD on different problem instances with $d=32$, $\sigma_{\widetilde{q}}^2=30$}
    \label{fig:Quantiles_d32_30}
\end{figure}

\begin{figure}[H]
    \centering
    \resizebox{\textwidth}{!}{%
    \includegraphics[]{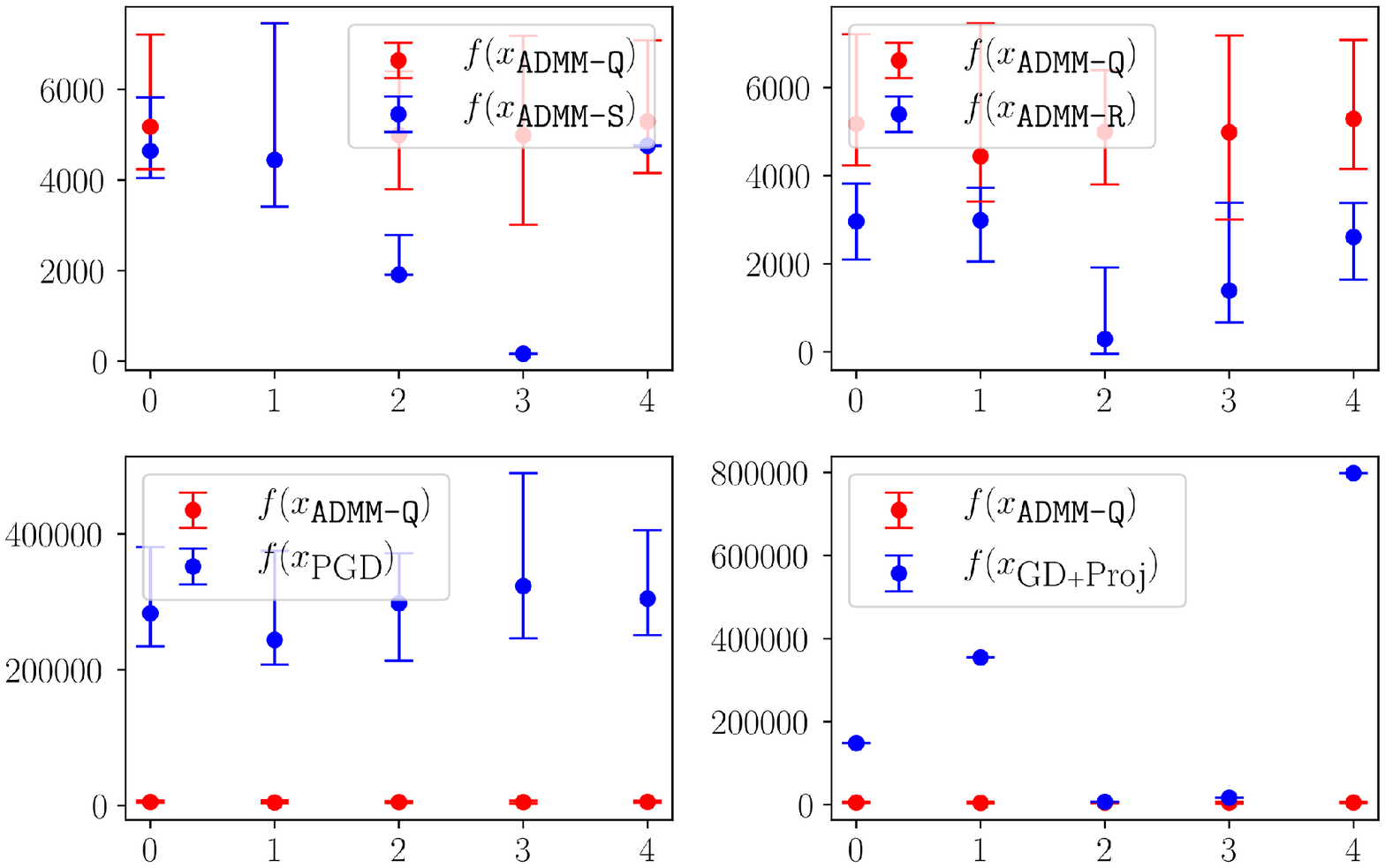}
    }
    \caption{Performance of \admmq, \admms, \admmr~ and PGD on different problem instances with $d=64$, $\sigma_{\widetilde{q}}^2=30$}
    \label{fig:Quantiles_d64_30}
\end{figure}

\begin{figure}[H]
    \centering
    \resizebox{\textwidth}{!}{%
    \includegraphics[]{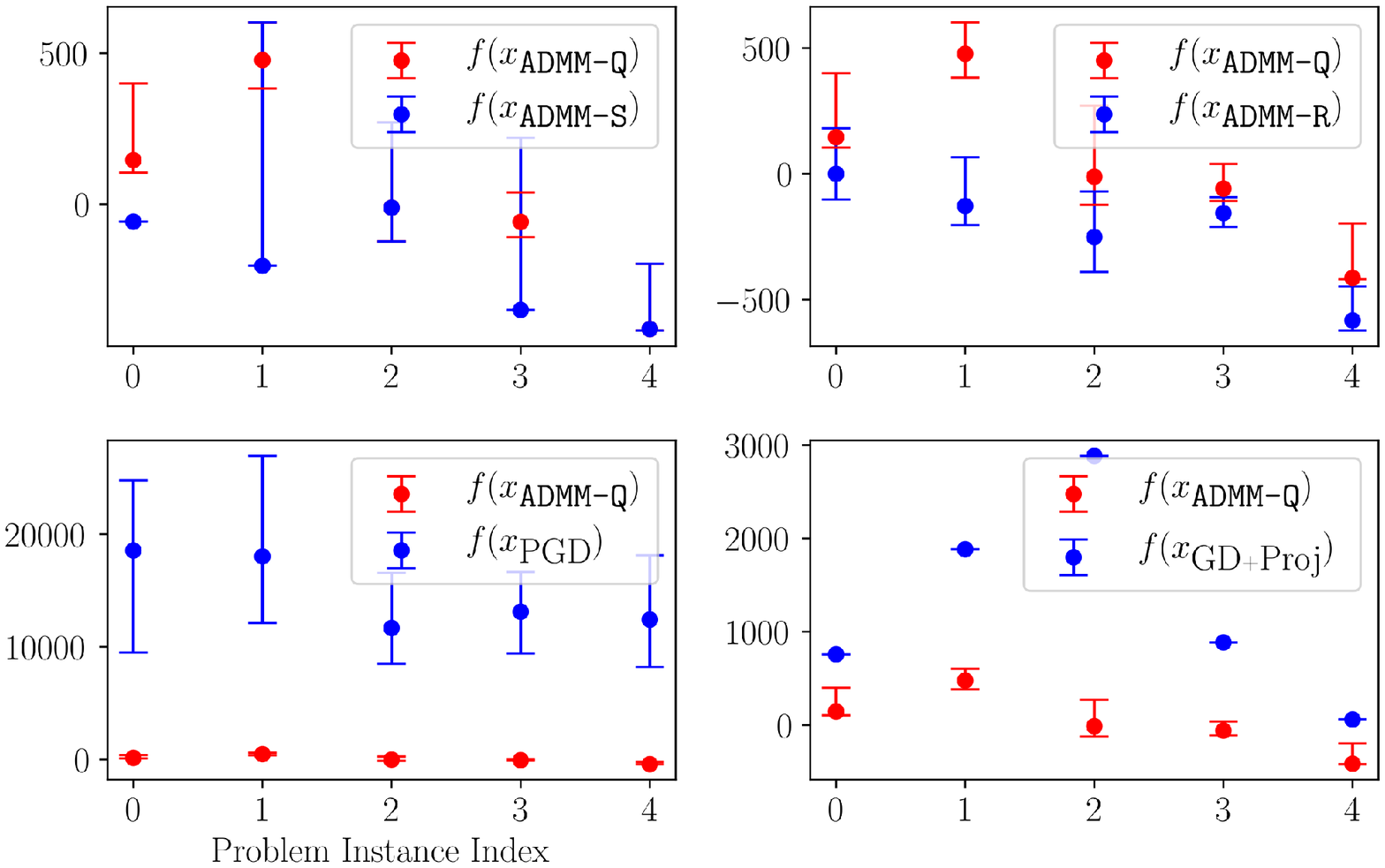}
    }
    \caption{Performance of \admmq, \admms, \admmr~ and PGD on different problem instances with $d=16$, $\sigma_{\widetilde{q}}^2=10$}
    \label{fig:Quantiles_d16_10}
\end{figure}

\begin{figure}[H]
    \centering
    \resizebox{\textwidth}{!}{%
    \includegraphics[]{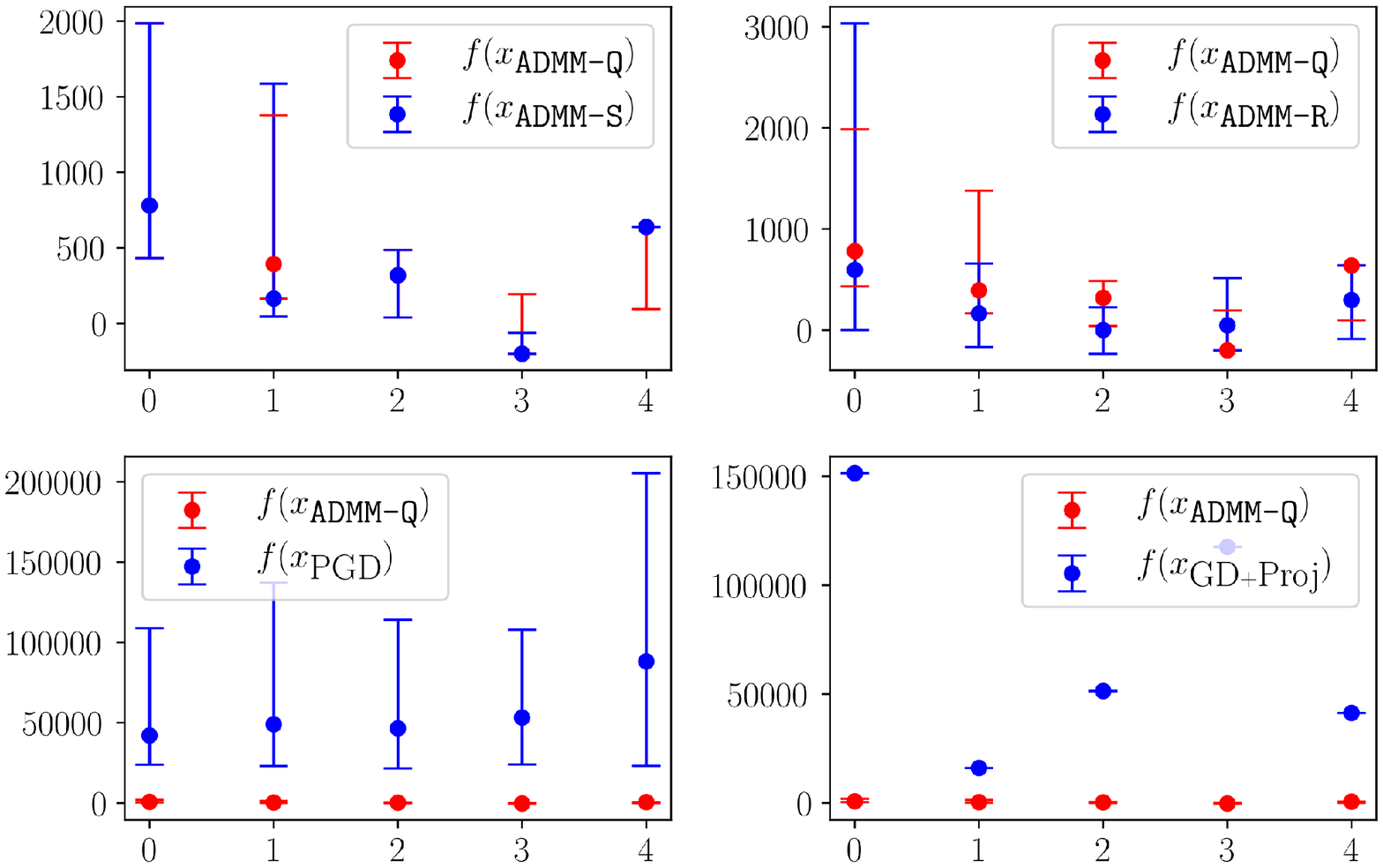}
    }
    \caption{Performance of \admmq, \admms, \admmr~ and PGD on different problem instances with $d=16$, $\sigma_{\widetilde{q}}^2=50$}
    \label{fig:Quantiles_d16_50}
\end{figure}

\begin{figure}[H]
    \centering
    \resizebox{\textwidth}{!}{%
    \includegraphics[]{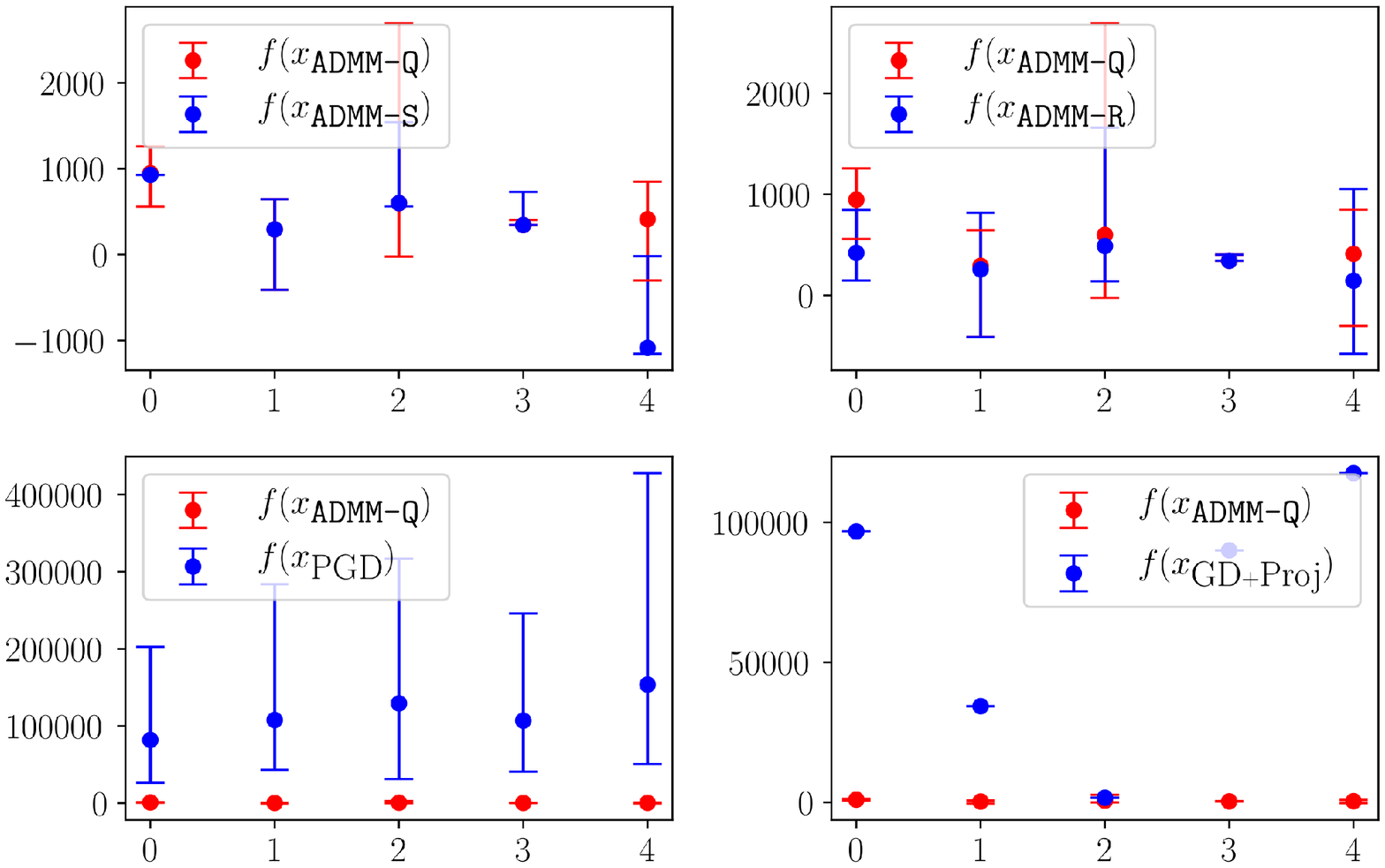}
    }
    \caption{Performance of \admmq, \admms, \admmr~ and PGD on different problem instances with $d=16$, $\sigma_{\widetilde{q}}^2=70$}
    \label{fig:Quantiles_d16_70}
\end{figure}

\begin{table}
\centering
\begin{tabular}{@{}ccc@{}}
\toprule
Algorithm             & \multicolumn{2}{c}{Hyper-parameters} \\ \midrule
\admmq                  & None & $\rho=10^{-k}, k\in \mathbb{Z}, -6\leq k \leq 2$            \\
\admms             & $\beta=10^{-5},10^{-4.5},10^{4},\ldots,10^{4.5},10^{5}$        & $\rho=10^{-k}, k\in \mathbb{Z}, -6\leq k \leq 2$        \\
\admmr & $p_i^{r}= 0.01, 0.1,0.3,0.5,0.7,0.9,0.99$            & $\rho=10^{-k}, k\in \mathbb{Z}, -6\leq k \leq 2$        \\ \bottomrule
\end{tabular}
\caption{Hyper-parameters used for \admmq, \admms~and \admmr}
\label{tab: hyper_parameters}
\end{table}

\section{Simulations on Neural Networks}
\label{app: simulation_nn_all}

Table~\ref{tab:Mnist-simple-network_pretrain_all} shows the performance of different algorithms on MNIST dataset. It suggests pre-training (with non-binarized weights) further improves the performance of ADMM-based methods. It is worth mentioning that, with pre-training, \admmq{} and its variants and even PGD algorithm are converging extremely fast, sometimes within as few as $3$ epochs. While in the presence of pre-training, and PGD and ADMM-based algorithms all work reasonably well, PGD is much more sensitive to initialization. In particular, omitting the pre-training phase drops the performance of PGD much more than the performance of ADMM-based methods. Table~\ref{tab:CIFAR-results_retrain_all} shows the results of the experiments on CIFAR-10 dataset. The observation is consistent with that from the MNIST dataset. Pre-training significantly improves the performance of the binarized models including both ADMM-based and PGD. Binarized models trained by ADMM-based algorithms with pre-training have comparable performance with the full precision model.

\begin{table}[H]
\centering
\begin{tabular}{@{}lc@{}}
\toprule
Algorithm      & Accuracy \\ \midrule
BinaryConnect~\cite{courbariaux2015binaryconnect}   &     $98.71\%$     \\
Full Precision & $98.87\pm 0.04\%$  \\
GD+Proj   &     $74.92\pm 4.83\%$     \\
PGD   & $92.73\pm 0.23\%$         \\
\admmq   & $98.21 \pm 0.16\%$         \\
\admmr   & \textcolor{black}{$97.78 \pm 0.23\%$}        \\ 
\admms   & $98.21 \pm 0.07\%$         \\ \midrule
PGD with pre-training  & $98.55\pm 0.05\%$         \\
\admmq~with pre-training &  $98.55\pm 0.04\%$         \\
\admmr~with pre-training &  $98.61\pm 0.06\%$          \\
\admms~with pre-training &  $98.57\pm 0.04\%$          \\
\bottomrule
\end{tabular}
\caption{Testing accuracies for MNIST dataset}
\label{tab:Mnist-simple-network_pretrain_all}
\end{table}

\begin{table}[H]
\centering
\begin{tabular}{@{}lc@{}}
\toprule
Algorithm      & Accuracy \\ \midrule
Progressive DNN~\cite{ye2019progressive}  &     $93.53\%$     \\
Full Precision & $93.06\%$  \\
GD+Proj   &     $9.86\%$     \\
PGD   & $63.53\%$         \\ 
\admmq   & $81.18\%$         \\ 
\admmr   & $84.87\%$         \\ 
\admms   & $84.72\%$         \\\midrule 
    PGD with pre-training  & $90.47\%$         \\
\admmq~ with pre-training & $90.42\%$         \\
\admmr~ with pre-training  &$90.46\%$          \\
\admms~ with pre-training  &$90.42\%$          \\
\bottomrule
\end{tabular}
\caption{Testing accuracies for CIFAR-10 dataset}
\label{tab:CIFAR-results_retrain_all}
\end{table}

\begin{table}[H]
\centering

\begin{tabular}{@{}ll@{}}
\toprule
Layer Type           & Shape               \\ \midrule
Dropout & $0.2$ \\
Fully Connected $+$ BatchNorm $+$ ReLU & $4096$ \\
Dropout & $0.5$ \\
Fully Connected $+$ BatchNorm $+$ ReLU & $4096$ \\ 
Dropout & $0.5$ \\
Fully Connected $+$ BatchNorm $+$ ReLU & $4096$ \\ 
Dropout & $0.5$ \\
Fully Connected $+$ BatchNorm & $10$ \\ \bottomrule
\end{tabular}
\caption{Model architecture for MNIST dataset.}
\label{tab:net-arch-mnist}

\end{table}

\begin{table}[H]
\centering
\begin{tabular}{@{}cclll@{}}
\toprule
\multicolumn{2}{c}{Algorithm} & \multicolumn{3}{c}{Parameter} \\ \midrule
\multicolumn{2}{c}{\multirow{3}{*}{GD (+ Proj)}} & Learning rate & $10^{-2}$ & $10^{-3}$ \\
\multicolumn{2}{c}{} & Epoch & 80 & 40 \\
\multicolumn{2}{c}{} & Batch-size & 512 & 512 \\ \midrule
\multicolumn{2}{c}{\multirow{3}{*}{PGD}} & Learning rate & $10^{-2}$ & $10^{-3}$ \\
\multicolumn{2}{c}{} & Epoch & 80 & 40 \\
\multicolumn{2}{c}{} & Batch-size & 512 & 512 \\ \midrule
\multicolumn{2}{c}{\multirow{4}{*}{\admmq}} & Learning rate & $10^{-2}$ & $10^{-3}$ \\
\multicolumn{2}{c}{} & Epoch & 80 & 40 \\
\multicolumn{2}{c}{} & Batch-size & 512 & 512 \\
\multicolumn{2}{c}{} & $\rho$ & $10^{-5}$ &  \\ \midrule
\multicolumn{2}{c}{\multirow{5}{*}{\admmr}} & Learning rate & $10^{-2}$ & $10^{-3}$ \\
\multicolumn{2}{c}{} & Epoch & 80 & 40 \\
\multicolumn{2}{c}{} & Batch-size & 512 & 512 \\
\multicolumn{2}{c}{} & $\rho$ & $10^{-5}$ &  \\
\multicolumn{2}{c}{} & $p^{r}_{i}$ & $0.99$ &  \\ \midrule
\multicolumn{2}{c}{\multirow{5}{*}{\admms}} & Learning rate & $10^{-2}$ & $10^{-3}$ \\
\multicolumn{2}{c}{} & Epoch & 80 & 40 \\
\multicolumn{2}{c}{} & Batch-size & 512 & 512 \\
\multicolumn{2}{c}{} & $\rho$ & $10^{-5}$ &  \\
\multicolumn{2}{c}{} & $\beta$ & $10^3$ &  \\ \midrule
\multirow{6}{*}{PGD with pre-training} & \multirow{3}{*}{Pre-training} & Learning rate & $10^{-2}$ & $10^{-3}$ \\
 &  & Epoch & 20 & 20 \\
 &  & Batch-size & 512 & 512 \\ \cmidrule(l){2-5} 
 & \multirow{3}{*}{Binariztion} & Learning rate & $10^{-2}$ & $10^{-3}$ \\
 &  & Epoch & 20 & 20 \\
 &  & Batch-size & 512 & 512 \\ \midrule
\multirow{7}{*}{\admmq~with pre-training} & \multirow{3}{*}{Pre-training} & Learning rate & $10^{-2}$ & $10^{-3}$ \\
 &  & Epoch & 20 & 20 \\
 &  & Batch-size & 512 & 512 \\ \cmidrule(l){2-5} 
 & \multirow{4}{*}{Binariztion} & Learning rate & $10^{-2}$ & $10^{-3}$ \\
 &  & Epoch & 20 & 20 \\
 &  & Batch-size & 512 & 512 \\
 &  & $\rho$ & $10^{-3}$ &  \\ \midrule
\multirow{8}{*}{\admmr~with pre-training} & \multirow{3}{*}{Pre-training} & Learning rate & $10^{-2}$ & $10^{-3}$ \\
 &  & Epoch & 20 & 20 \\
 &  & Batch-size & 512 & 512 \\ \cmidrule(l){2-5} 
 & \multirow{5}{*}{Binariztion} & Learning rate & $10^{-2}$ & $10^{-3}$ \\
 &  & Epoch & 20 & 20 \\
 &  & Batch-size & 512 & 512 \\
 &  & $\rho$ & $10^{-3}$ &  \\
 &  & $p^{r}_{i}$ & $0.3$ &  \\ \midrule
\multirow{8}{*}{\admms~with pre-training} & \multirow{3}{*}{Pre-training} & Learning rate & $10^{-2}$ & $10^{-3}$ \\
 &  & Epoch & 20 & 20 \\
 &  & Batch-size & 512 & 512 \\ \cmidrule(l){2-5} 
 & \multirow{5}{*}{Binariztion} & Learning rate & $10^{-2}$ & $10^{-3}$ \\
 &  & Epoch & 20 & 20 \\
 &  & Batch-size & 512 & 512 \\
 &  & $\rho$ & $10^{-3}$ &  \\
 &  & $\beta$ & $10^3$ &  \\ \bottomrule
\end{tabular}
\caption{Training parameters for MNIST dataset.}
\label{tab:training_para_mnist}
\end{table}


\begin{table}[H]
\centering
\begin{tabular}{@{}ccllllll@{}}
\toprule
\multicolumn{2}{c}{Algorithm} & \multicolumn{6}{c}{Parameter} \\ \midrule
\multicolumn{2}{c}{\multirow{3}{*}{GD (+ Proj)}} & Learning rate & $10^{-2}$ & $10^{-3}$ & $10^{-4}$ &  &  \\
\multicolumn{2}{c}{} & Epoch & 100 & 100 & 100 &  &  \\
\multicolumn{2}{c}{} & Batch-size & 512 & 512 & 512 &  &  \\ \midrule
\multicolumn{2}{c}{\multirow{3}{*}{PGD}} & Learning rate & $10^{-3}$ & $10^{-3}$ & $10^{-3}$ & $10^{-3}$ & $10^{-4}$ \\
\multicolumn{2}{c}{} & Epoch & 200 & 200 & 200 & 200 & 400 \\
\multicolumn{2}{c}{} & Batch-size & 512 & 512 & 512 & 512 & 512 \\ \midrule
\multicolumn{2}{c}{\multirow{4}{*}{\admmq}} & Learning rate & $10^{-3}$ & $10^{-3}$ & $10^{-3}$ & $10^{-3}$ & $10^{-4}$ \\
\multicolumn{2}{c}{} & Epoch & 200 & 200 & 200 & 200 & 400 \\
\multicolumn{2}{c}{} & Batch-size & 512 & 512 & 512 & 512 & 512 \\
\multicolumn{2}{c}{} & $\rho$ & $10^{-5}$ & $10^{-4}$ & $10^{-3}$ & $10^{-2}$ & $10^{-2}$ \\ \midrule
\multicolumn{2}{c}{\multirow{5}{*}{\admmr}} & Learning rate & $10^{-3}$ & $10^{-3}$ & $10^{-3}$ & $10^{-3}$ & $10^{-4}$ \\
\multicolumn{2}{c}{} & Epoch & 200 & 200 & 200 & 200 & 400 \\
\multicolumn{2}{c}{} & Batch-size & 512 & 512 & 512 & 512 & 512 \\
\multicolumn{2}{c}{} & $\rho$ & $10^{-5}$ & $10^{-4}$ & $10^{-3}$ & $10^{-2}$ & $10^{-2}$ \\
\multicolumn{2}{c}{} & $p^{r}_{i}$ & $0.975$ &  &  &  &  \\ \midrule
\multicolumn{2}{c}{\multirow{5}{*}{\admms}} & Learning rate & $10^{-3}$ & $10^{-3}$ & $10^{-3}$ & $10^{-3}$ & $10^{-4}$ \\
\multicolumn{2}{c}{} & Epoch & 200 & 200 & 200 & 200 & 400 \\
\multicolumn{2}{c}{} & Batch-size & 512 & 512 & 512 & 512 & 512 \\
\multicolumn{2}{c}{} & $\rho$ & $10^{-5}$ & $10^{-4}$ & $10^{-3}$ & $10^{-2}$ & $10^{-2}$ \\
\multicolumn{2}{c}{} & $\beta$ & $0.05\rho$ &  &  &  &  \\ \midrule
\multirow{6}{*}{PGD with pre-training} & \multirow{3}{*}{Pre-training} & Learning rate & $10^{-2}$ & $10^{-3}$ & $10^{-4}$ &  &  \\
 &  & Epoch & 100 & 100 & 100 &  &  \\
 &  & Batch-size & 512 & 512 & 512 &  &  \\ \cmidrule(l){2-8} 
 & \multirow{3}{*}{Binariztion} & Learning rate & $10^{-3}$ & $10^{-4}$ & $10^{-5}$ &  &  \\
 &  & Epoch & 250 & 250 & 250 &  &  \\
 &  & Batch-size & 512 & 512 & 512 &  &  \\ \midrule
\multirow{7}{*}{\admmq~with pre-training} & \multirow{3}{*}{Pre-training} & Learning rate & $10^{-2}$ & $10^{-3}$ & $10^{-4}$ &  &  \\
 &  & Epoch & 100 & 100 & 100 &  &  \\
 &  & Batch-size & 512 & 512 & 512 &  &  \\ \cmidrule(l){2-8} 
 & \multirow{4}{*}{Binariztion} & Learning rate & $10^{-3}$ & $10^{-4}$ & $10^{-5}$ &  &  \\
 &  & Epoch & 250 & 250 & 250 &  &  \\
 &  & Batch-size & 512 & 512 & 512 &  &  \\
 &  & $\rho$ & $0.05\rho$ &  &  &  &  \\ \midrule
\multirow{8}{*}{\admmr~with pre-training} & \multirow{3}{*}{Pre-training} & Learning rate & $10^{-2}$ & $10^{-3}$ & $10^{-4}$ &  &  \\
 &  & Epoch & 100 & 100 & 100 &  &  \\
 &  & Batch-size & 512 & 512 & 512 &  &  \\ \cmidrule(l){2-8} 
 & \multirow{5}{*}{Binariztion} & Learning rate & $10^{-3}$ & $10^{-4}$ & $10^{-5}$ &  &  \\
 &  & Epoch & 250 & 250 & 250 &  &  \\
 &  & Batch-size & 512 & 512 & 512 &  &  \\
 &  & $\rho$ & $10^{-2}$ &  &  &  &  \\
 &  & $p^{r}_{i}$ & $0.975$ &  &  &  &  \\ \midrule
\multirow{8}{*}{\admmr~with pre-training} & \multirow{3}{*}{Pre-training} & Learning rate & $10^{-2}$ & $10^{-3}$ & $10^{-4}$ &  &  \\
 &  & Epoch & 100 & 100 & 100 &  &  \\
 &  & Batch-size & 512 & 512 & 512 &  &  \\ \cmidrule(l){2-8} 
 & \multirow{5}{*}{Binariztion} & Learning rate & $10^{-3}$ & $10^{-4}$ & $10^{-5}$ &  &  \\
 &  & Epoch & 250 & 250 & 250 &  &  \\
 &  & Batch-size & 512 & 512 & 512 &  &  \\
 &  & $\rho$ & $10^{-2}$ &  &  &  &  \\
 &  & $\beta$ & $0.02\rho$ &  &  &  &  \\ \bottomrule
\end{tabular}
\caption{Training parameters for CIFAR-10 dataset.}
\label{tab:training_para_cifar-10}
\end{table}
{
\color{black}
\section{Link to the Code}
\label{app:code}
Codes are available at \url{https://github.com/optimization-for-data-driven-science/ADMM-Q}.
}

\end{document}